\documentclass{amsart}
  \usepackage[top=1in, bottom=1in, left=1.375in, right=1.375in]{geometry}
 \usepackage{amsmath,amssymb}
 \usepackage{algpseudocode}
 \usepackage{algorithm}
 \usepackage{algorithmicx}
 \usepackage{caption}
 \usepackage{subcaption}
 \usepackage{amsthm}
 \numberwithin{equation}{section}
 \usepackage{amsfonts}
 \usepackage{graphicx}
 \usepackage{hyperref}

\usepackage{amsmath}
\usepackage{amssymb}
\usepackage{amsthm}
\usepackage{amsmath}
\usepackage{setspace}
\usepackage{xcolor}
\usepackage{fancyhdr}
\usepackage{amssymb}
\usepackage{amsthm}
\usepackage{listings}
    \usepackage{cite}
    \usepackage{tabularx}
    \usepackage{graphicx}
    \usepackage{epstopdf}
    \usepackage{epsfig}
    \usepackage{float}
    \usepackage{ listings}
    \usepackage{appendix}
 \theoremstyle{plain}
 \newtheorem{thm}{Theorem}[section]
 \newtheorem{prop}[thm]{Proposition}
 \newtheorem{lem}[thm]{Lemma}
 
 \theoremstyle{definition}

 \theoremstyle{remark}
 \newtheorem{remark}[thm]{Remark}

 \let\pa=\partial
 \let\al=\alpha
 
 \let\d=\delta
 \let\g=\gamma
 \let\e=\varepsilon

 \let\lam=\lambda

 \let\f=\frac
 \let \les = \lesssim
 \let\om=\omega

 \let \vp = \varphi
 
\let\B = \Big
 
 \let\Lam=\Lambda
 
 \let\Om=\Omega
 \let\td = \tilde

 \let\teq \triangleq
 
 \let\pa=\partial

 \def\cF{{\mathcal F}}

 \def\cL{{\mathcal L}}

 \def\la{\langle}
 \def\ra{\rangle}
\def\lt{\left}
\def\rt{\right}

 \newcommand{\beq}{\begin{equation}}
 \newcommand{\eeq}{\end{equation}}
  \newcommand{\bal}{\begin{aligned} }
  \newcommand{\eal}{\end{aligned}}
 \newcommand{\ben}{\begin{eqnarray}}
 \newcommand{\een}{\end{eqnarray}}
 \newcommand{\beno}{\begin{eqnarray*}}
 \newcommand{\eeno}{\end{eqnarray*}}

 \newcommand{\R}{\mathbb{R}}

 \newcommand{\CL}{\mathcal{L}}

\newcommand{\sgn}{\mathrm{sgn}}

 \author{Jiajie Chen}
 \address{Applied and Computational Mathematics, California Institute of Technology, Pasadena, CA 91125, USA}
 \date{August 23, 2019}

\title[Generalized Constantin-Lax-Majda Equation with Dissipation]{
Singularity Formation and Global Well-Posedness for the Generalized Constantin-Lax-Majda Equation with Dissipation}
 
\begin{document}
 \maketitle

 \begin{abstract}
We study a generalization due to De Gregorio and Wunsch et.al. of the Constantin-Lax-Majda equation (gCLM) on the real line 
\[
\omega_t + a u \omega_x  = u_x \omega  - \nu \Lambda^{\gamma} \omega, \quad u_x =  H \omega ,
\]
where $H$ is the Hilbert transform and $\Lambda = (-\partial_{xx})^{1/2}$. We use the method in \cite{chen2019finite} to prove finite time self-similar blowup for $a$ close to $\frac{1}{2}$ and $\gamma=2$ by establishing nonlinear stability of an approximate self-similar profile. For $a>-1$, we discuss several classes of initial data and establish global well-posedness and an one-point blowup criterion for different initial data. For $a\leq-1$, we prove global well-posedness for gCLM with critical and supercritical dissipation.
\end{abstract}

 \section{Introduction}
Constantin, Lax and Majda \cite{CLM85} introduced an one-dimensional equation (CLM) that models the effect of vortex stretching in a three-dimensional incompressible fluid. They established finite time singularity formation in their model. In the CLM equation, the convection term is missing, which we now know has stabilizing effect in 3D incompressible flow 
\cite{hou2008dynamic,lei2009stabilizing}. Inspired by their work, De Gregorio \cite{DG90,DG96} proposed to include a convection term to the CLM equation to model the effects of both convection and the vortex stretching. Okamoto, Sakajo, and Wunsch \cite{OSW08} introduced an one-parameter family of models by modeling the strength of the convection term to study the interplay of convection and vortex stretching. Wunsch \cite{wunsch2011generalized} further studied 
this one-parameter family of models with dissipation. The full model then becomes
\beq\label{eq:DG}
\bal
\om_t + a u \om_x& = u_x \om  - \nu \cL \om\\
        u_x &=  H \om ,
\eal
\eeq
where $a$ is a real parameter, $H$ is the Hilbert transform and $\CL $ is typically some dissipative operator, such as a full or fractional Laplacian $\Lam^{\g}, \g \in [0, 2]$.  
If $a= 0, \nu =0$, \eqref{eq:DG} reduces to the CLM equation. If $a = 1, \nu =0$, it is the De Gregorio equation and $\nu \neq 0$ corresponds to the viscous version. If $a=-1$, this is the equation (up to a differentiation) considered by Cordoba, Cordoba and Fontelos \cite{cordoba2005formation} for the surface quasi-geostrophic (SQG) equation. There are various 1D models proposed in the literature. We refer to \cite{Elg17,sverak2017certain,kiselev2018} for excellent surveys of other 1D models for the 3D Euler equations, surface quasi-geostrophic equation and other equations. Throughout this paper, we call \eqref{eq:DG} the generalized Constantin-Lax-Majda equation (gCLM).

In this paper, we study the singularity formation of \eqref{eq:DG} (gCLM) for a range of $a$ and establish several results about the global well-posedness for entire range of $a$ on the real line.

\subsection{Previous works}
The local well-posedness and the BKM type blowup criterion for \eqref{eq:DG} on the circle have been established in \cite{OSW08,wunsch2011generalized}. Similar results on the real line can be obtained using a standard method. 

For singularity formation or global well-posedness of \eqref{eq:DG}, we first review some results of the inviscid case $\nu =0$. One important feature of the De Gregorio model and its full generalization 
is that it captures the competition between the convection term and the vortex stretching term. It is not hard to see that when $a < 0$, the convection effect would work together with the stretching effect to produce a singularity. Indeed, Castro and Cordoba \cite{Cor10} proved the finite time blow-up for all $a < 0$ based on a Lyapunov functional argument. See also \cite{chen2019finite} for a simple proof on the circle. 
For the special case $a = -1$, finite time blowup was established earlier by Cordoba et.al. \cite{cordoba2005formation}. For $a =0$, finite time singularity was established in \cite{CLM85}.  For $a>0$, there are competing nonlocal stabilizing effect due to the convection and the destabilizing effect due to the vortex stretching, which are of the same order in terms of scaling. 
Due to this competition, the same Lyapunov functional argument in \cite{Cor10} would fail to prove a finite time singularity. 
We remark that the stabilizing effect of convection has been studied by Hou et.al. \cite{hou2008dynamic,lei2009stabilizing}. 

In \cite{Elg17}, Elgindi and Jeong constructed smooth self-similar profile for small $a$ and 
$C^{\al}$ self-similar profile for general $a$.
Recently, the author, Hou and Huang \cite{chen2019finite} established the finite time blowup of the De Gregorio equation, i.e. \eqref{eq:DG} with $a=1, \nu = 0$, with $C_c^{\infty}$ initial data
on the real line by proving that an approximate steady state is nonlinearly stable. 
Self-similar blowup for \eqref{eq:DG} with $\nu =0$ and the entire range of parameter $a$ on 
the real line or the circle for H\"older continuous initial data with compact support were also established in the same paper \cite{chen2019finite}. See also \cite{Elg19} for similar results. The global well-posedness of the De Gregorio equation on the circle with smooth initial data is still open. We remark that Sverak et. al. \cite{sverak2017certain} and \cite{lei2018constantin} proved that the equilibrium $ A \sin(2(x-x_0))$ of the De Gregorio equation on the circle is nonlinearly stable. 

For \eqref{eq:DG} with dissipation, singularity formation or global well-posedness is only known for some special $a$. When $a = 0, \nu >0$, Schochet \cite{Schonet1986} established finite time singularity for the viscous CLM equation with $\cL \om = -\om_{xx}$. In this case, an explicit solution of \eqref{eq:DG} can be obtained. 
When $a=-1$, \eqref{eq:DG} becomes the Cordoba-Cordoba-Fontelos (CCF) equation, which has been studied extensively. Finite time blowup with supercritical dissipation $\cL = \Lam^{\g}, \g < \f{1}{2}$ was established in \cite{hou2008dynamic,silvestre2014transport} and the global well-posedness with subcritical and critical dissipation were established in \cite{cordoba2005formation,dong2008well}. For some global solution to the supercritical CCF, see \cite{ferreira2018global}. When $a \leq -2$ is even and $\cL = \Lam$, the global well-posedness for small initial data was established in \cite{wunsch2011generalized}.

\subsection{Scaling and the critical dissipation} Suppose that $\cL = \Lam^{\g}$ in \eqref{eq:DG} for some $\g \in [0,2]$. Then the solution of \eqref{eq:DG} enjoys the following scaling property: if $\om(x, t)$ is a solution of \eqref{eq:DG}, then for any $\lam > 0$,
\[
\om_{\lam}(x, t) \teq \lam^{\g} \om( \lam x, \lam^{\g} t)
\]
is also a solution of \eqref{eq:DG}. 

For \eqref{eq:DG} with $a >-1$, there is no coercive conserved quantity or a-priori estimate
for general initial data, which makes it very difficult to prove global well-posedness. 
We will show that for several classes of initial data, $|| \om||_{L^1}$ is conserved. In these cases, a simple scaling analysis shows that $\cL = \Lam$ corresponds to the critical dissipation. 

For \eqref{eq:DG} with $a \leq -1$, we will show that the equation possesses a-priori 
$L^{|a|}$ estimate, i.e. $ || \om(t, \cdot) ||_{L^{|a|}} \leq || \om_0 ||_{L^{|a|}} $, which makes $\Lam^{\g}$ with $\g = |a|^{-1}$ the critical dissipation with respect to the natural scaling of the equation.

\subsection{Assumption on $\cL$ when $a>-1$}\label{sec:Lass}
We consider $\cL$ which is slightly subcritical (compared to $\Lam$) and define it below
\beq\label{eq:visop0}
\CL \om(x) \teq P.V.\int  \f{ \om(x) - \om(x-y) }{ m(|y|)  |y|^2 } dy ,
\eeq
where $m:  [0, + \infty)  \to [0, + \infty)$ satisfies the following assumptions

\quad (a) $m$ is a non-decreasing function.

\quad (b) Slightly subcritical dissipation:
 \beq\label{eq:dispm0}
\lim_{r \to 0+}  r m(r)^{1/2} \int_r^1 \f{1}{s^2 m(s)} d s = +\infty , \quad \lim_{r \to 0^+} m(r) = 0.
\eeq

For example, for the fractional Laplacian $\Lambda^{\al}, \al \in (1,2)$, $m$ is $m(r) = r^{\al - 1}$.
The dissipative operator with condition \eqref{eq:dispm0} can be weaker than any fractional Laplacian $\Lambda^{\al } $, $\al > 1$ and approaches the critical dissipation $\Lambda$ up to a logarithm term. In fact, $m(r)$ can behave like $  | \ln r |^{-e}$ for all sufficiently small $r$ and $\e > 0$.

\subsection{Main results}  The main results of this paper are the following. 

 \begin{thm}[Finite time blow-up for $a$ close to $\f{1}{2}$]\label{thm:blowup_vis}
 Consider \eqref{eq:DG} with $\cL \om = -\pa_{xx} \om$. There exists $\d   > 0 $ such that for $a \in  ( \f{1}{2} - \d, \f{1}{2} + \d)$, $ 0\leq \nu \leq 1$, \eqref{eq:DG_vis0} develops a self-similar singularity in finite time for some $C_c^{\infty}$ initial data.
 \end{thm}

\begin{remark}
$\nu =0$ corresponds to the inviscid case. If the dissipative operator is replaced by fractional Laplacian $\Lam^{\g}$ with $\g \in [0,2]$, one can apply similar analysis to obtain finite time blowup. We focus on the full Laplacian for simplicity. In \cite{chen2019finite}, the inviscid self-similar profile of \eqref{eq:DG} with $a=0$ was proved to be nonlinearly stable. Using an argument similar to that in the proof of Theorem \ref{thm:blowup_vis}, one can also prove finite time blowup of \eqref{eq:DG}with $\cL = \Lam^{\g}, \g \in [0, 1)$ when $a$ is sufficiently close to $0$ 
\end{remark}

For \ref{eq:DG} with $a > -1$, there is no coercive conserved quantity for general initial data. 
We restrict to the following
initial data with odd symmetry or satisfying some sign property. 

\quad Class 1. $\om_0$ has a fixed sign for all $x \in \R$,

\quad Class 2. $\om_0$ is odd and $\om \geq0 $ for $x> 0$,

\quad Class 3. $\om_0$ is odd and $\om \leq 0$ for $x >0$.

It is not difficult to verify that the above symmetry and sign property in each class are preserved by \eqref{eq:DG} during the evolution. Note that any initial data that have odd or even symmetry and a fixed sign for $x>0$ belong to one of the above classes. Our second main result is the following global well-posedness and one-point blowup criterion. 
\begin{thm}[Global well-posedness and blowup criterion for slightly subcritical gCLM]
\label{thm:GWP}
Consider the dissipative operator $\cL$ defined in \eqref{eq:visop0} satisfying assumption \eqref{eq:dispm0}. Suppose that $a> -1 ,\nu >0$ in \eqref{eq:DG}, $\om_0 \in L^1 \cap H^1$. 
 For $\om_0$ in class 1 and 2, there exists a unique solution of \eqref{eq:DG} globally in time.  For $\om_0$ in class 3, the unique local in time solution cannot be extended beyond $T >0$ if and only if
\beq\label{eq:ux0BKM}
\int_0^T u_x(0, t) dt = +\infty.
\eeq
\end{thm}

\begin{remark}
 The initial data constructed in \cite{chen2019finite} that lead to finite time blowup of the inviscid gCLM \eqref{eq:DG} 
 all fall in class 3. For $\om_0$ in class 3, it is easy to verify that $u_x(t,0) \geq 0$.
 The blowup criterion \eqref{eq:ux0BKM} is an analogue of the celebrated Caffarelli-Kohn-Nirenberg partial regularity theorem \cite{caffarelli1982partial} for the 3D axisymmetric Navier-Stokes equations in the sense that singularity must occur on the symmetry axis, if it exists. \eqref{eq:ux0BKM} is sharp due to the blowup result in Theorem \ref{thm:blowup_vis}. For the De Gregorio equation, i.e. $a=1 , \nu = 0$ in \eqref{eq:DG}, Lei et.al. \cite{lei2018constantin} proved the GWP 
for initial data $\om_0$ in the class 1 with $ |\om_0|^{1/2} \in H^1$. It is not clear if similar result holds for $a \neq 1$ with initial data in this class.
\end{remark}

For $a \leq -1$, there is a-priori $L^{|a|}$ estimate for \eqref{eq:DG}. We prove the global well-posedness of \eqref{eq:DG} in the whole range $a\leq -1$ with critical or supercritical dissipation.

\begin{thm}[Global well-posedness for critical or supercritical gCLM]\label{thm:GWPc}
Suppose that $\cL = \Lam^{\g}, a \leq -1, \nu >0$ in \eqref{eq:DG} and the initial data $\om_0 \in L^{|a|} \cap H^1$. If $ \g \in [ |a|^{-1}, 2 ] $, \eqref{eq:DG_vis2} has a unique global solution. Moreover, for $a < -1$, there exists $\g_1 = \g_1( || \om_0||_{L^{|a|}}, || \om||_{H^1}  , a)  \in (0, |a|^{-1} ) $, such that for each $\g \in [\g_1, |a|^{-1} ]	$, \eqref{eq:DG_vis2} has a unique global solution.
\end{thm}

\begin{remark}
For $a<-1$, $\Lam^{\g}$ with $\g_1 \leq \g < |a|^{-1}$ corresponds to the supercritical dissipation. When $a=-1$, the global well-posedness with critical dissipation was first proved in \cite{dong2008well} using modulus of continuity. The novelty of our approach is to establish a-priori $L^{|a| + \e}$ estimate for sufficiently small $\e > 0$ depending on the norm of the initial data. 
\end{remark}

\begin{remark}
In the presence of dissipation, Theorem \ref{thm:GWPc} shows that for \eqref{eq:DG} with a fix dissipative operator $\cL = \Lam^{\g}, \g >0, \nu >0$ and $a \leq -1$, as $|a|$ becomes larger and satisfies $| a| \g \geq 1$, the equation becomes globally well-posed. This confirms the stabilizing effect of the convection term studied in e.g. \cite{hou2008dynamic,lei2009stabilizing,hou2018global}. 
\end{remark}

\subsection{Organization of the paper}
In Section \ref{sec:vis_blowup}, we construct the self-similar profile for the inviscid gCLM \eqref{eq:DG} with $a=\f{1}{2}$ and prove Theorem \ref{thm:blowup_vis} by showing that such profile is nonlinearly stable. In Section \ref{sec:GWP}, we first perform the $L^1$ estimate for several classes of initial data and then prove Theorem \ref{thm:GWP} using the nonlinear maximal principle \cite{constantin2012nonlinear}. In Section \ref{sec:GWP}, we prove Theorem \ref{thm:GWPc} by establishing a-priori $L^{|a| + \e}$ estimate of the solution. In the Appendix, we prove several Lemmas about the Hilbert transform and the fractional Laplacian and some useful estimates.

\vspace{0.1in}
\paragraph{\bf{Notations}}

We use $\la \cdot , \cdot \ra$ to denote the inner product on $R$, i.e.
\[
\la f , g \ra  \teq \int_{\R} f g dx .
\]
Without specification, the domain of the integral is the whole real line, i.e. 
$
\int f  = \int_{\R} f(x) dx,$  $\iint f = \iint_{\R \times \R} f(x, y) dx dy.$ We use $C $ to denote absolute constant and $C(A,B,..,Z)$ to denote constant depending on $A,B ,..,Z$. These constants may vary from line to line. We use $C_i$ to denote constant which does not vary. We use the notation $A \les B$ if there is some absolute constant $C > 0$ with $A \leq CB$. 
The upper bar notation is reserved for the approximate profile, e.g. $\bar{\om}$.

\section{Finite time blowup for $a $ close to $1/2$}\label{sec:vis_blowup}
In this Section, we prove Theorem \ref{thm:blowup_vis}. Consider \eqref{eq:DG} 
with dissipation equal to $\cL = -\pa_{xx}$
\beq\label{eq:DG_vis0}
\bal
\om_t + a u \om_x &= u_x \om  +  \nu  \om_{xx} , \quad u_x = H \om .
\eal
\eeq
 
Firstly, we study the inviscid problem, i.e. $\nu = 0$. Consider an ansatz of self-similar solution 
$\om = (T-t)^{c_{\om}} \Om( \f{ x}{(T-t)^{c_l}})$. For $a=1/2$, Plugging this ansatz in \eqref{eq:DG_vis0} with $\nu =0$ yields
\[
( c_l x  + \f{1}{2} U ) \Om_x = (c_{\om} + U_x) \Om,
\]
where $U_x = H \Om$. Surprisingly, it has the following analytic self-similar solution 
\beq\label{eq:vis_solu}
\Om = \f{-2bx}{ (x^2 + b^2)^2} ,   \quad  U_x = \f{b^2 - x^2} {  (b^2 + x^2)^2 }, 
\quad U = \f{x}{b^2 + x^2},   \quad  c_l = \f{1}{3},  \quad  c_{\om} = -1,
\eeq
where $b = \sqrt{3/8}$. To verify that $(\Om, c_l)$ solves the self-similar equation, we compute
\[
\bal
( c_l x  + \f{1}{2} U )  \Om_x \cdot \f{1}{\Om} + c_{\om} - U_x
&= ( \f{x }{3}  + \f{1}{2} \f{x}{b^2 + x^2} ) 
\lt( \f{-2 b}{ (b^2 + x^2)^2} + \f{2bx \cdot 4 x}{ (b^2 + x^2)^3}  \rt)
\f{ (b^2 + x^2)^2}{-2bx}    + 1 - \f{b^2 - x^2}{ (b^2 + x^2)^2} \\
&= ( \f{1 }{3}  + \f{1}{2} \f{1}{b^2 + x^2} )  
\lt( 1  - \f{4 x^2}{x^2 + b^2} \rt) + 1  -  \f{b^2 - x^2}{ (b^2 + x^2)^2} \\
& = \f{4}{3} + \f{1}{x^ 2 + b^2} (\f{1}{2} -\f{4}{3}x^2) 
- \f{ b^2 + x^2}{ (x^2 + b^2)^2}   = \f{4}{3}  \f{b^2}{x^2 + b^2} - \f{1 }{2} \f{1}{x^2 + b^2} = 0, 
\eal
\]
where we have used $ b^2 = \f{3}{8}$ to obtain the last identity. Hence,  plugging the self-similar solution into \eqref{eq:DG_vis0} implies finite time self-similar blowup 
for $a=1/2$ and $\nu = 0$. In this self-similar blowup, the spatial blowup scaling is $c_l = 1/3$, 
if we add the diffusion term, such term is asymptotically small compared to the nonlinear term in the equation of the self-similar variables. In our later analysis, we will treat the diffusion term as a small perturbation to the nonlinear part, especially the vortex stretch term $u_x \om$. 

\subsection{Dynamic rescaling formulation}\label{sec:dsform} \

To prove finite time self-similar blowup, we use the strategy developed in \cite{chen2019finite}. We reformulate the problem of proving finite time self-similar singularity into the problem of establishing the nonlinear stability of an approximate self-similar profile using the dynamic rescaling equation.

Let $ \om(x, t), u(x,t)$ be the solutions of the original equation \eqref{eq:DG_vis0}. It is easy to show that 
\beq\label{eq:rescal1}
  \td{\om}(x, \tau) = C_{\om}(\tau) \om(   C_l(\tau) x,  t(\tau) ), \quad   \td{u}(x, \tau) = C_{\om}(\tau) 
C_l(\tau)^{-1}u(C_l(\tau) x, t(\tau))
\eeq
are the solutions to the dynamic rescaling equations
 \beq\label{eq:DGdy00}
\bal
\td{\om}_{\tau}(x, \tau) + ( c_l(\tau) x + a \td{u} ) \td{\om}_x(x, \tau)  &=   c_{\om}(\tau) \td{\om} + \td{u}_x \td{ \om  } + \nu(\tau ) \td{\om}_{xx}, \quad \td{u}_x &= H \td{\om} ,
\eal
\eeq
where  $\nu( \tau ) = C_l(\tau )^{-2} C_{\om}(\tau), \  t(\tau) = \int_0^{\tau} C_{\om}(\tau) d\tau $, 
\beq\label{eq:DGdy01}
\bal
  C_{\om}(\tau) = \exp\lt( \int_0^{\tau} c_{\om} (s)  d \tau\rt) C_{\om}(0), \ C_l(\tau) = \exp\lt( \int_0^{\tau} -c_l(s) d s    \rt) C_l(0) .
\eal
\eeq
We have the freedom to choose the initial rescaling factors $C_{\om}(0), C_l(0)$ and the time-dependent scaling parameters $c_l(\tau)$ and $c_{\om}(\tau)$ according to some normalization conditions. After we determine $C_{\om}(0), C_l(0)$ and the normalization conditions for $c_l(\tau)$ and $c_{\om}(\tau)$, 
\eqref{eq:DGdy00} is completely determined and the solution of 
\eqref{eq:DGdy00} is equivalent to that of the original equation using the scaling relationship described in \eqref{eq:rescal1}-\eqref{eq:DGdy01}, as long as $c_l(\tau)$ and $c_{\om}(\tau)$ remain finite.

We remark that a similar dynamic rescaling formulation was employed in \cite{mclaughlin1986focusing,  landman1988rate} to study the nonlinear Schr\"odinger (and related) equation. In some literature, this formulation is called the modulation technique. It has been very efficient to describe the formation of singularities for many problems like the nonlinear wave equation \cite{merle2015stability}, the nonlinear heat equation \cite{merle1997stability}, the generalized KdV equation \cite{martel2014blow}, and other dispersive problems. It has recently been applied to prove singularity formation in fluid dynamics \cite{elgindi2019finite}.

If there exists $C>0$ such that for any $\tau > 0$, $c_{\om}(\tau) \leq -C <0$ and the solution $\td{\om}$ is nontrivial, e.g. $ || \td{\om}(\tau, \cdot) ||_{L^{\infty}} \geq c >0$ for all $\tau >0$, we then have 
\[
\bal
C_{\om}(\tau) &\leq e^{-C\tau}, \ t(\infty) \leq \int_0^{\infty}  e^{-C \tau } d \tau =C^{-1} <+ \infty \; ,
\eal
\]
and that $| \om(   C_l(\tau) x,  t(\tau) ) | = C_{\om}(\tau)^{-1}  |\td{\om}(x, \tau) | 
\geq e^{C\tau} |\td{\om}(x, \tau) | $  blows up at finite time $T = t(\infty)$.
If $(\td{\om}_{\tau}, c_{l}(\tau), c_{\om}(\tau))$ converges to a steady state $(\om_{\infty}, c_{l, \infty}, c_{\om,\infty})$ of \eqref{eq:DGdy00} as $\tau \to \infty$, one can verify that 
\[
\om(x, t) = \f{1}{1 - t} \om_{\infty}\lt( \f{x}{ (1 - t)^{-c_{l, \infty} / c_{\om, \infty} }  }  \rt)
\]
is a self-similar solution of \eqref{eq:DG_vis0}.

To simplify our presentation, we still use $t$ to denote the rescaled time in the rest of the paper and drop $\td{\cdot}$ in \eqref{eq:DGdy00}, which leads to
\beq\label{eq:DGdy_vis}
\bal
\om_t + (c_l x + au) \om_x &= (c_{\om} + u_x) \om + \nu(t)  \om_{xx} , \quad u_x = H\om ,
\eal
\eeq
where
\beq\label{eq:nu}
\nu(t) =  C_l(t)^{-2} C_{\om}(t)  \nu = \exp \lt( \int_0^{t} (c_{\om}(s) + 2 c_l(s) ) ds  \rt)
C_l(0)^{-2} C_{\om}(0)  \nu .
\eeq

\subsection{Approximate steady state and normalization condition} \

For given $a$ close to $\f{1}{2}$, we use the profile \eqref{eq:vis_solu} to construct the approximate steady state
\beq\label{eq:vis_appro}
\bal
&\bar{\om} = \f{-2bx}{ (x^2 + b^2)^2} , \quad \bar{u}_x = H\bar{\om} =  \f{b^2 - x^2} {  (b^2 + x^2)^2 }, \quad \bar{u} = \f{x}{x^2 + b^2}, \\
&\bar{c}_l = \f{1}{3}  - (a-\f{1}{2}) \bar{u}_x(0), \quad  \bar{c}_{\om}(t) = -1 - \f{\nu(t) \bar{\om}_{xxx}(0) }{\bar{\om}_x(0)}, \quad b = \sqrt{ \f{3}{8}}.
\eal
\eeq
We remark that $\bar{c}_{\om}(t)$ is time-dependent and $\bar{\om}$ is odd. We consider the equation of any perturbation $(\om, u_x, c_l, c_{\om})$ with $\om(0, \cdot)$ being odd around the above approximate steady state. Clearly, the equation preserves the property that $\om(t, \cdot)$ is odd during the evolution. The equation for the perturbation reads
\beq\label{eq:vislin}
\bal
\om_t  & = -   ( \bar{c}_l x +a \bar{u} ) \om_x + ( \bar{c}_{\om} + \bar{u}_x ) \om  + ( u_x + c_{\om} )  \bar{\om} - (  c_l x+ au ) \bar{\om}_x + \nu(t)  \om_{xx}  + N(\om) + F(\bar{\om},t ) \\
& \teq \cL(\om, \nu(t)) +   N(\om) + F(\bar{\om},t ),
\eal
\eeq
where the nonlinear and error terms are given by
\beq\label{eq:vis_NF}
\bal
N(\om) &=  - (c_l x + au) \om_x + (c_{\om} + u_x) \om, \\
  F(\bar{\om}, t) & = - ( a- \f{1}{2} )(\bar{u} - \bar{u}_x(0) x) \bar{\om}_x + \nu(t)  \bar{ \om}_{xx} .
 \eal
\eeq
We choose the following normalization condition for the perturbation
\beq\label{eq:normal_vis}
\bal
 & c_l(t)  = -a u_x(0)  , \quad   c_{\om}(t) = -u_x(0) - \nu(t) \f{\om_{xxx}(t, 0)}{ \bar{\om}_x(0) }. \\
\eal
\eeq
If the initial perturbation satisfies $\om_x(0,0) = 0$, $\om_x(t,0)$ remains $0$. In fact, if $\om_x(t,0) =0$, plugging \eqref{eq:vis_appro} and \eqref{eq:normal_vis} into \eqref{eq:DGdy_vis}, we get 
\[
\bal
\f{d}{dt} \om_x(t,0) &= \f{d}{dt} (\om_x(t,0) + \bar{\om}_x(0))
 = ( c_{\om}(t) + \bar{c}_{\om}  + u_x(0) + \bar{u}_x(0) )  (\om_x(t,0) + \bar{\om}_x(0)) \\
 &\quad - (c_l(t) + \bar{c}_l + a(u_x(0) + \bar{u}_x(0))  ) (\om_x(t,0) + \bar{\om}_x(0))
+ \nu(t) ( \om_{xxx}(t, 0) + \bar{\om}_{xxx}(0) )  \\
& = \bar{\om}_x(0) \lt(
 - 1 + \bar{u}_x(0)  - \nu(t) \f{\om_{xxx}(t,0) + \bar{\om}_{xxx}(0) }{\bar{\om}_x(0)} - \lt( \f{1}{3} + \f{1}{2} \bar{u}_x(0) \rt) \rt)  \\
 &\quad +\nu(t) ( \om_{xxx}(t, 0) + \bar{\om}_{xxx}(0) ) = 0,
\eal
\]
where we have used $\bar{u}_x(0) = \f{8}{3}$. Hence, $\om_x(t,0)=0$ is preserved.

We choose the following weights
\beq\label{eq:wg_vis}
\vp = \f{ (x^2 + b^2)^3 }{2b x^4} = -\f{1}{\bar{\om}} \f{x^2 + b^2}{x^3},  \quad \psi = \f{(x^2 + b^2)^3}{2b} = -\f{x(x^2 + b^2)}{\bar{\om}},
\eeq
and will perform weighted $L^2$ and weighted $H^4$ estimates to establish the nonlinear stability. Since $\om$ vanishes at least quadratically near $x=0$ for smooth $\om$, $\la\om^2, \vp \ra$ is well-defined. The $H^4$ estimate is to control $\om_{xxx}(t,0)$ appeared in $c_{\om}(t,0)$. 

In the following discussion, we assume $\om \in L^2(\vp) \cap H^4(\psi)$.

\subsection{Linear estimate}
We perform the weighted $L^2$ estimate as follows
\beq\label{eq:linl2}
\bal
\f{1}{2} \f{d}{dt} \la \om^2, \vp \ra & = \la - ( \bar{c}_l x +a \bar{u} ) \om_x +    ( \bar{c}_{\om} + \bar{u}_x ) \om, \om \vp \ra
 + \la ( u_x + c_{\om} )  \bar{\om} , \om  \vp \ra  - \la  (a  u + c_l x ) \bar{\om}_x ,\om \vp\ra   \\
 &  + \la \nu(t)  \om_{xx},  \om \vp \ra  + \la  N(\om) ,\om \vp\ra + \la F(\bar{\om}) , \om \vp \ra  \teq I + II + III +D + N + F .
\eal
\eeq

For $I$, we use integration by parts to get
\beq\label{eq:vislin_1}
I = \B\la \f{1}{2\vp} (  ( \bar{c}_l x +a  \bar{u} ) \vp )_x + ( \bar{c}_{\om} + \bar{u}_x ) ,\om^2 \vp \B\ra .
\eeq

Denote
\beq\label{eq:ux0}
\td{u}(x) = u(x) - u_x(0) x ,\quad \td{u}_x(x) \teq u_x(x) - u_x(0) .
\eeq

Recall the definition of $\vp$ in  \eqref{eq:wg_vis}. A direct calculation yields 
\[
\bal
II &=  \la ( u_x + c_{\om} )  \bar{\om} , \om  \vp \ra
 = - \B\la \lt(\td{u}_x - \nu(t) \f{\om_{xxx}(t, 0)}{ \bar{\om}_x(0) }  \rt), \om \f{x^2 + b^2}{x^3} \B\ra \\
& = -\B\la  \td{u}_x\om, \f{1}{x} + \f{b^2}{x^3} \B\ra + \nu(t)\f{\om_{xxx}(t, 0)}{ \bar{\om}_x(0) }  \B\la \om, \f{x^2 + b^2}{x^3} \B\ra.
\eal
\]
Using the cancellation \eqref{lem:vel3} and \eqref{lem:vel4},  we get
\[
II = - \f{\pi}{2} u_x^2(0) + \nu(t)\f{\om_{xxx}(t, 0)}{ \bar{\om}_x(0) }  \B\la \om, \f{x^2 + b^2}{x^3} \B\ra \leq  \nu(t)\f{\om_{xxx}(t, 0)}{ \bar{\om}_x(0) }  \B\la \om, \f{x^2 + b^2}{x^3} \B\ra.
\]
Note that
\[
u_x(0) =-\f{1}{\pi}\int_R \f{\om}{y} dy , \quad u_{xxx}(0) = -\f{1}{\pi} \int_R \f{\om_{xx}}{x} dx = -\f{2}{\pi} \int_R \f{\om}{x^3} dx,
\]
where we can use integration by parts twice for the second integral since $\om_x(0) =0$ by normalization condition and $\om_{xx}(0) = 0$ by odd-even symmetry. Hence
\beq\label{eq:vislin_2}
II \leq - \nu(t)\f{\om_{xxx}(t, 0)}{ \bar{\om}_x(0) }
\f{\pi}{2}\lt( 2u_x(0)  + b^2 u_{xxx}(0)\rt) .
\eeq

For $III$, we first use \eqref{eq:normal_vis} and \eqref{eq:ux0} to rewrite $c_l x + au$
\[
c_l x + a u = a u - a u_x(0) x = a \td{u}.
\]

Using the elementary inequality $|xy| \leq x^2 + \f{y^2}{4}$ yields
\[
\bal
| III | &  = | \la  (a  u + c_l x ) \bar{\om}_x ,\om \vp\ra  |=  a | \la  \td{u} \bar{\om}_x ,\om \vp \ra |  
\leq \B\la   \td{u}^2,  \f{ b}{x^4}    \B\ra + \f{a^2}{4 b} \B\la \om^2,  (\bar{\om}_x \vp)^2 x^4 \B\ra,
\eal
\]
where $b = \sqrt{\f{3}{8}}$ is the same constant used in \eqref{eq:vis_appro}. We can use the Hardy inequality \eqref{eq:hd1} with $p=2$ to control the velocity
\beq\label{eq:vislin_3}
| III| \leq \B\la \om^2,  \f{4 b}{9}\f{1}{x^2} \B\ra
+  \f{a^2}{4 b} \B\la \om^2,  (\bar{\om}_x \vp)^2 x^4 \B\ra
= \B\la  \om^2, \f{4 b}{9}\f{1}{x^2}  + \f{a^2}{4 b }  (\bar{\om}_x \vp)^2x^4  \B\ra.
\eeq

For $D$, we separate the singular and less singular part of the weight $\vp$ defined in \eqref{eq:wg_vis}
\[
D=   \nu(t) \B\la \om_{xx}, \om \f{ (x^2 + b^2)^3 }{2b x^4}  \B\ra
= \nu(t) \B\la \om_{xx}, \om  \f{ (x^2 + b^2)^3  - b^6 }{2b x^4}  \B\ra   + \nu(t) b^6 \la  \om_{xx}, \om x^{-4} \ra
\teq P_1 + P_2.
\]
Note that $\om(0)=  \om_x(0) = \om_{xx}(0) = 0$. Near the origin, we have
\[
\om = O(|x|^3) , \om_{x} = O(|x|^2) .
\]
Thus we can use integration by parts to estimate $P_1$
\[
\bal
P_1 &= - \nu(t) \B\la  \om_x^2 ,\f{ (x^2 + b^2)^3  - b^6 }{2b x^4} \B\ra
- \nu(t) \B\la \om_x , \om \lt( \f{ (x^2 + b^2)^3  - b^6 }{2b x^4} \rt)_x \B\ra \\
&\leq - \nu(t) \B\la \om_x , \om \lt( \f{ (x^2 + b^2)^3  - b^6 }{2b x^4} \rt)_x \B\ra
= \f{\nu(t)}{2} \B\la \om^2, \lt( \f{ (x^2 + b^2)^3  - b^6 }{2b x^4} \rt)_{xx} \B\ra \les \f{\nu(t)}{2} \la \om^2, \vp \ra,
\eal
\]
where we have used the following estimate to get the last inequality
\[
\B| \lt(\f{ (x^2 + b^2)^3 - b^6 }{2b x^4} \rt)_{xx}  \B| \les 1 + x^{-4} \les \vp .
\]
For $P_2$, since $\om(0) = \om_x(0) = \om_{xx}(0) =0 $, we can use Hardy inequality to get
\[
\la \om^2, x^{-6} \ra \les \la \om_x^2, x^{-4} \ra \les \la \om_{xx}^2, x^{-2} \ra \les || \om_{xxx}||_2^2.
\]

Using the Cauchy-Schwartz inequality and the above estimate, we get
\[
P_2 \leq \nu(t) \la \om_{xx}^2, x^{-2}\ra^{1/2} \la \om^{2}, x^{-6} \ra^{1/2} \les \nu(t) ||\om_{xxx}||_2^2.
\]

Combining the estimate of $P_1,P_2$, we yield
 \beq\label{eq:vislin_4}
D = P_1 +P_2 \les \nu(t) ( \la \om^2, \vp \ra +  || \om_{xxx}||_2^2  ).
 \eeq

Collecting the estimate \eqref{eq:vislin_1}, \eqref{eq:vislin_2} and \eqref{eq:vislin_3} and \eqref{eq:vislin_4}, we derive
\beq\label{eq:vislin_L20}
\bal
\f{1}{2}\f{d}{dt} \la \om^2, \vp \ra
&\leq
\B\la \f{1}{2\vp} (  ( \bar{c}_l x +a  \bar{u} ) \vp )_x + ( \bar{c}_{\om} + \bar{u}_x ) ,\om^2 \vp \B\ra
+ \B\la  \om^2,  
\f{4 b}{9}\f{1}{x^2}  + \f{a^2}{4 b }  (\bar{\om}_x \vp)^2x^4 
\B\ra   \\
&- \f{\pi}{2}  \nu(t)\f{\om_{xxx}(t, 0)}{ \bar{\om}_x(0) }
\lt( 2u_x(0)  + b^2 u_{xxx}(0)\rt)  \\
&+ C\nu(t) ( \la \om^2, \vp \ra +  || \om_{xxx}||_2^2  ) + \la  N(\om) ,\om \vp\ra + \la F(\bar{\om}, t) , \om \vp \ra .\\
\eal
\eeq
We remark that the second and the third line in \eqref{eq:vislin_L20} can be arbitrary small by choosing $\nu(0)$  and $a - \f{1}{2}$ to be small. 
Notice that we have analytic formulas \eqref{eq:vis_appro} and \eqref{eq:wg_vis} for $\bar{u}_x, \bar{c}_l , \bar{u}, \bar{c}_{\om},\bar{\om},\vp$ in the first line of \eqref{eq:vislin_L20}. Using the pointwise estimate \eqref{eq:vis_Sneg2} in Lemma \ref{lem:vis_dp} whose proof is elementary, we can further estimate the quantities in the first line on the right hand side of \eqref{eq:vislin_L20} and yield
\beq\label{eq:vislin_L2}
\bal
\f{1}{2}\f{d}{dt} \la \om^2, \vp \ra
&\leq  - (\f{1}{2} -C( |a- \f{1}{2}| +  \nu(t) ) ) \la \om^2, \vp \ra - \f{\pi}{2} \nu(t)\f{\om_{xxx}(t, 0)}{ \bar{\om}_x(0) }  \lt( 2u_x(0)  + b^2 u_{xxx}(0)\rt)  \\
&+ C\nu(t) ( \la \om^2, \vp \ra +  || \om_{xxx}||_2^2  ) +\la  N(\om) ,\om \vp\ra + \la F(\bar{\om}, t) , \om \vp \ra .\\
\eal
\eeq

\subsection{Weighted $H^4$ estimate}
Taking the $x$ derivative on both sides of \eqref{eq:vislin} four times and then multiplying 
$\pa_x^4\om \psi$ give
\beq\label{eq:linh4}
\bal
\f{1}{2}\f{d}{dt} \la (\pa_x^4\om)^2, \psi \ra &=  \la - \pa_x^4(( \bar{c}_l x +a \bar{u} ) \om_x)
+ \pa_x^4 (( \bar{c}_{\om} + \bar{u}_x ) \om ) , \pa_x^4 \om \psi\ra  
+  \la \pa_x^4 (( u_x + c_{\om} )  \bar{\om} ), \pa_x^4 \om \psi \ra \\
& - \la \pa_x^4( (  c_l x+ au ) \bar{\om}_x), \pa_x^4\om \psi\ra  
+ \nu(t)  \la \pa_x^6 \om , \pa_x^4 \om\psi \ra  \\
&+ \la \pa_x^4 N(\om), \pa_x^4\om \psi \ra + \la \pa_x^4 F(\bar{\om},t ) , \pa_x^4\om \psi \ra  \teq I + II + III + D_2 + N_2 + F_2.
\eal
\eeq
Denote  $f^{(k)} = \pa_x^k f$. 
For the terms in each inner product, we only keep track the perturbation term with the same or higher order than $\om^{(4)}$, i.e. $\om^{(4)},\om^{(5)}, \om^{(6)}$, $u^{(5)}$. We use the terminology \textit{l.o.t} to denote the lower order term.
For $I$, we have
\[
I = \la - (\bar{c}_l x +a \bar{u} ) \om^{(5)} - 4(\bar{c}_l + a\bar{u}_x) \om^{(4)} + (\bar{c}_{\om} + \bar{u}_x) \om^{(4)}, \om^{(4)} \psi \ra + \la l.o.t, \ \om^{(4)} \psi \ra.
\]
Using integration by parts, we get
\beq\label{eq:vislin_51}
I = \B\la \f{1}{2 \psi} (  (\bar{c}_l x + a \bar{u})\psi)_x  - (4 \bar{c}_l + 4 a \bar{u}_x - \bar{c}_{\om}-\bar{u}_x), (\om^{(4)})^2\psi \B\ra + \la l.o.t , \ (\om^{(4)})^2\psi \ra .
\eeq

Notice that we have analytic formulas \eqref{eq:vis_appro} and \eqref{eq:wg_vis} for $\bar{u}_x, \bar{c}_l , \bar{u}, \bar{c}_{\om},\psi$. Using the pointwise estimate \eqref{eq:vis_Sneg3} in Lemma \ref{lem:vis_dp}, 
we get
\beq\label{eq:vislin_5}
I \leq \B\la - \f{1}{3}+ C( |a-\f{1}{2}| +\nu(t) ) , (\om^{(4)})^2 \psi \B\ra + \la l.o.t., \ \om^{(4)} \psi\ra
\eeq
for some universal constant $C$.

For $II$, we have
\beq\label{eq:vislin_60}
\bal
II &=  \la \pa_x^4 (( u_x + c_{\om} )  \bar{\om} ), \pa_x^4\om \psi \ra = \la u^{(5)}\bar{\om},  \om^{(4)} \psi \ra + \la l.o.t. ,\ \om^{(4)} \psi \ra \\
&= - \la u^{(5)}  \om^{(4)} ,  x(x^2 + b^2)\ra + \la l.o.t. ,\ \om^{(4)} \psi \ra .
\eal
\eeq
Next, we show that the first term vanishes. Applying the cancellation \eqref{lem:vel5} with $
(u_{xx}, \om_x) $ replaced by $( u^{(5)} , \om^{(4)} ) $, we get
\beq\label{eq:vislin_61}
-  b^2 \la u^{(5)}   \om^{(4)} , x \ra = 0.
\eeq
Using integration by parts yields 
\[
 H(\om^{(4)}x)(0)  = - \f{1}{\pi} \int_R \om^{(4)} dx = 0 , \quad 
 H(\om^{(4)}x^2)(0)  = -\f{1}{\pi} \int_R \om^{(4)} x dx = \f{1}{\pi}
 \int_R \om^{(3)}  dx = 0.
\]
Note that $H(\om^{(4)})(x) = u^{(5)}$. Applying Lemma \ref{lem:vel0}, we obtain 
\[
\bal
 H( \om^{(4)} x) =x (H \om^{(4)} ) = x  u^{(5)}(x)  ,\quad  H( \om^{(4)}x^2) = x H(\om^{(4)} x)  = x^2 (H \om^{(4)} ) = x^2 u^{(5)}(x),
\eal
\]
Moreover, we have $( x^2u^{(5)} )(0) = (x^2 \om^{(4)})(0) = 0$. Therefore, applying Lemma \ref{lem:tric} with $\om = x^2 \om^{(4)}$, we yield 
\beq\label{eq:vislin_62}
\bal
 - \la u^{(5)}(x) \om^{(4)}, x^3 \ra  &= - \la   H(x^2 \om^{(4)})  \cdot ( x^2 \om^{(4)}), x^{-1} \ra  
 = \pi H(  H(x^2 \om^{(4)})  \cdot ( x^2 \om^{(4)}) )(0) \\
 & = \f{\pi}{2}   \lt(    H(x^2 \om^{(4)}) (0) ^2  - ( x^2 \om^{(4)})(0)^2 \rt) = 0.
 \eal
\eeq

Combining \eqref{eq:vislin_60},  \eqref{eq:vislin_61} and  \eqref{eq:vislin_62}, we can derive
\beq\label{eq:vislin_6}
II  =   \la l.o.t. ,\ \om^{(4)} \psi \ra .
\eeq
For $III$, it only involves $u^{(i)}$ for $i=0,1,2,3,4$. Hence
\beq\label{eq:vislin_7}
III  = \la l.o.t, \ \om^{(4)} \psi \ra.
\eeq

Recall the definition of $\psi$ in \eqref{eq:wg_vis}. 
For $D_2$, we use integration by parts to get
\beq\label{eq:vislin_8}
\bal
D_2 & = \nu(t)  \la  \pa_{xx} \om^{(4)} , \om^{(4)} \psi \ra = - \nu(t) \la  \om^{(5)}, \om^{(5)} \psi \ra
-\nu(t) \la \om^{(5)}, \om^{(4)} \psi_x \ra \\
& = - \nu(t) \la  \om^{(5)}, \om^{(5)} \psi \ra + \f{\nu(t)}{2} \la \om^{(4)}, \om^{(4)} \psi_{xx} \ra
\leq \f{\nu(t)}{2} \la \om^{(4)}, \om^{(4)} \psi_{xx} \ra \les \nu(t) \la (\om^{(4)} )^2, \psi \ra,
\eal
\eeq
where we have used the following estimate to derive the last inequality
\[
|\psi_{xx} | = \B| \lt( \f{(x^2 + b^2)^3}{2b} \rt)_{xx} \B| \les 1 + x^4 \les \psi.
\]

For all the lower order terms in $I, II, III$, we can use interpolation between
$\la \om^2, \vp\ra$ and $\la \om^{(4)}, \psi \ra$ to control them. Since $\bar{\om}$ is not singular and decays very fast $x^{-3}$ for large $x$, all the coefficients, i.e. $\bar{\om},\bar{u}, \psi$, in the interaction
$\la l.o.t. , \ \om^{(4)} \psi \ra $
are bounded by some absolute constant. 
We estimate a representative term. For 
 $\la \pa_x^4 (c_l x + au) \bar{\om}_x, \om^{(4)}\psi\ra$ in III,  we can estimate it as follows
 \[
|\la \pa_x^4(c_l x + au) \bar{\om}_x, \om^{(4)}\psi\ra | = a |\la u^{(4)} \bar{\om}_x, \om^{(4)} \psi \ra|
 \les || u^{(4)}||_2 || \bar{\om}_x \psi^{1/2}||_{\infty} \la (\om^{(4)})^2, \psi \ra^{1/2}.
 \]
Using \eqref{eq:vis_appro} and \eqref{eq:wg_vis}, we have
 \[
| \bar{\om}_x  \psi^{1/2} | \les \f{1}{1+x^4} (1 + x^2)^{3/2} \les 1.
 \]
 It follows
 \[
 \bal
\la \pa_x^4 (c_l x + au) \bar{\om}_x, \om^{(4)}\psi\ra &\les  || u^{(4)}||_2 ||  \la (\om^{(4)})^2, \psi \ra^{1/2}
 \les || \om^{(3)} \|_2 \la (\om^{(4)})^2, \psi \ra^{1/2} \\
 &\les || \om||^{1/4}_2 || \om^{(4)}||^{3/4}_2   \la (\om^{(4)})^2, \psi \ra^{1/2}
 \les \la \om^2, \vp \ra^{1/8} \la (\om^{(4)} )^2, \psi \ra^{7/8} .\\
 \eal
 \]
 Using $\e$-Young's inequality, we further obtain 
\[
 \la \pa_x^4(c_l x + au) \bar{\om}_x, \om^{(4)}\psi\ra \leq \e \la (\om^{(4)} )^2, \psi \ra + C(\e) \la \om^2, \vp\ra.
 \]
 for some constant $C(\e)$ depending on $\e$. Similarly, for all terms in $\la l.o.t., \ \om^{(4)}\psi \ra$, we have
 \beq\label{eq:vislin_9}
| \la l.o.t. , \ \om^{(4)} \psi \ra | \leq \e \la (\om^{(4)})^2, \psi \ra + C(\e) \la \om^2, \vp \ra .
 \eeq
Finally, collecting all the estimates \eqref{eq:vislin_5}, \eqref{eq:vislin_6}, \eqref{eq:vislin_7}, \eqref{eq:vislin_8} and \eqref{eq:vislin_9}, we obtain the weighted $H^4$ estimate up to the nonlinear term and the error
\[
\bal
\f{1}{2} \f{d}{dt} \la (\om^{(4)})^2, \psi \ra  &\leq
\B\la - \f{1}{3} + C( |a-1/2| +\nu(t) ) , (\om^{(4)})^2 \psi \B\ra + \e \la (\om^{(4)})^2, \psi \ra  \\
&+ C(\e) \la \om^2, \vp \ra  + \nu(t) \la (\om^{(4)})^2, \psi \ra  + \la  \pa^4_x N(\om), \om^{(4)} \psi \ra + \la  \pa_x^4 F(\bar{\om},t ) , \om^{(4)} \psi \ra ,\\
\eal
\]
where $C(\e)$ is some constant depending on $\e$. We choose $\e = \f{1}{12}$ so that $C(\e) = C(\f{1}{12})$ is a universal constant. The above inequality can be simplified as follows
\beq\label{eq:vislin_H4}
\bal
\f{1}{2} \f{d}{dt} \la (\om^{(4)})^2, \psi \ra & \leq
\B\la -\f{1}{4} + C_1 |a- \f{1}{2}| +\nu(t) ) , (\om^{(4)})^2 \psi \B\ra +  C_1\la \om^2, \vp \ra  \\
& + \la \pa_x^4 N(\om), \om^{(4)} \psi \ra + \la \pa_x^4 F(\bar{\om},t ) , \om^{(4)} \psi \ra ,\\
\eal
\eeq
for some universal constant $C_1 > 0$.

\subsection{Nonlinear stability}
Now we construct the main energy
\beq\label{eg}
E^2(t) = \la \om^2 ,\vp \ra + \mu \la (\om^{(4)})^2 ,\psi \ra , \quad \mu = \f{1}{8 C_1},
\eeq
where $C_1$ is chosen in \eqref{eq:vislin_H4}. With this fix energy, we can control the $L^{\infty}$ norm of many terms e.g.
\[
 || \pa^k_{x}\om||_{\infty}
\les || \om||_2 + || \om^{(4)}||_2 \les E(t), \quad || \pa_x^k u_x||_{\infty} \les E(t),
\]
for $k = 0,1,2,3$. We remark that $\om_{xxx}(0)$ appears in $c_{\om}(t)$. 

Recall the nonlinear term
$N(\om)$ defined in \eqref{eq:vis_NF}. Next, we control the nonlinear parts $N, N_2$ in the $L^2, H^4$ estimates \eqref{eq:linl2}, \eqref{eq:linh4}
\[
N = \la  N(\om) ,\om \vp\ra, \quad  N_2 =\la \pa_x^4 N(\om),   \pa_x^4\om \psi \ra.
\]

Consider a representative term 
\[
P = \la a u  \pa_x^{4} \om_x , \pa_x^4\om \psi \ra
\]
in the nonlinear part of the $H^4$ estimate. Using integration by parts yields
\[
|P| = \f{ |a|}{2} \B\la  \f{1}{\psi } (u\psi)_x , (\pa_x^4\om)^2 \psi \B\ra
\les  \B\la |u_x| + \f{|u|}{x} \f{ |x \psi_x|}{\psi}, (\pa_x^4\om)^2 \psi \B\ra
\les || u_x||_{\infty} \la (\pa_x^4\om)^2, \psi \ra \les E^3(t),
\]
where we have used the calculation in \eqref{eq:xpsi} to yield $ \f{|x\psi_x| }{\psi}\les 1 $ in the second inequality. We can estimate other terms in the nonlinear part $N, N_2$ similarly. Hence, we prove
\beq\label{eq:visnlin_N1}
|\la  N(\om) ,\om \vp\ra |\les (1 + \nu(t) )E^3(t) , \quad |\la \pa_x^4 N(\om), \om^{(4)} \psi \ra |\les (1 + \nu(t)) E^3(t) .
\eeq
Recall the error term in the weighted $L^2$ and weighted $H^4$ estimate
\beq\label{eq:visnlin_F1}
 \la F(\bar{\om}, t) , \om \vp \ra, \quad  \la \pa_x^4 F(\bar{\om},t ) , \pa_x^4 \om \psi \ra ,
\eeq
where $F(\bar{\om},t)$ is defined (\eqref{eq:vis_NF}). Using the formulas of $\bar{\om}, \bar{u}$ in \eqref{eq:vis_appro}, one can easily verify that 
\[
| F(\bar{\om}, t)| \les  (  | a - \f{1}{2} | + \nu(t)  )  \min( |x| , |x|^{-3} ).
\]
We use the Cauchy-Schwartz and the Hardy inequality
$\la \om^2, x^{-6} \ra \les || \om^{(3)}||_2^2 $
to get
\beq\label{eq:visnlin_F2}
\bal
 &\la F(\bar{\om}, t) , \om \vp \ra = \B\la F(\bar{\om},t), \om \f{ (x^2 + b^2)^3  }{2b x^4} \B\ra
\les \B\la  |F(\bar{\om},t) \om|, x^2 + \f{1}{x^4} \B\ra  \\
 \les &\la F(\bar{\om},t)^2,  x^4 \ra^{1/2}
\la \om^2,  1 \ra^{1/2} + \la F(\bar{\om}, t)^2, x^{-2} \ra^{1/2} \la \om^2, x^{-6}\ra^{1/2} \\
\les & (  | a - \f{1}{2}  | + \nu(t) )  \lt( || \om||_{L^2} + || \om^{(3)} ||_{L^2}   \rt) \les 
 (  | a - \f{1}{2}  | + \nu(t) ) E(t),
 \eal
\eeq
where in the last inequality we have used the interpolation between $\la \om^2, \vp \ra$ and $\la \om^{(4)}, \psi \ra$ to control $||\om^{(3)}||_{L^2}$. For the second error term in \eqref{eq:visnlin_F1}, the weight $\psi$ is not singular and the error $ \pa_x^4 F(\bar{\om},t ) $ decays sufficiently fast. Hence, we obtain 
\beq\label{eq:visnlin_F3}
 \la \pa_x^4 F(\bar{\om},t )  , \om^{(4)}\psi \ra \les \la ( \pa_x^4 F(\bar{\om},t ))^2, \psi \ra^{\f{1}{2}}
 \la ( \om^{(4)})^2, \psi \ra^{ \f{1}{2}} \les    (  | a - \f{1}{2} | + \nu(t)  )  E(t).
\eeq
For the remaining terms 
\[
- \f{\pi}{2}  \nu(t)\f{\om_{xxx}(t, 0)}{ \bar{\om}_x(0) }  \lt( 2u_x(0)  + b^2 u_{xxx}(0)\rt)  +  \nu(t)  || \om_{xxx}||_2^2   \\
\]
in the weighted $L^2$ estimate \eqref{eq:vislin_L2}, each term $\om_{xxx}(t,0), u_x(0), u_{xxx}(0) ,||\om_{xxx}||_2$ can be bounded by an interpolation between $\la \om^2, \vp \ra$ and $\la \om^{(4)}, \psi \ra$. Hence, we have
\beq\label{eq:visnlin_R}
- \f{\pi}{2}  \nu(t)\f{\om_{xxx}(t, 0)}{ \bar{\om}_x(0) }  \lt( 2u_x(0)  + b^2 u_{xxx}(0)\rt) 
+ \nu(t)  || \om_{xxx}||_2^2  \les \nu(t) E^2(t).
\eeq
Combining \eqref{eq:vislin_L2}, \eqref{eq:vislin_H4}, \eqref{eq:visnlin_N1}, \eqref{eq:visnlin_F2},\eqref{eq:visnlin_F3} and \eqref{eq:visnlin_R}, we prove
\beq\label{eq:vis_boot1}
\bal
\f{1}{2} \f{d}{dt} E^2(t) &= \f{1}{2} \f{d}{dt}\la \om^2 ,\vp \ra + \mu \la (\om^{(4)})^2
\leq - \lt( \f{1}{2}-C( |a-1/2| +  \nu(t) ) \rt) \la \om^2, \vp \ra + C\nu(t) E^2(t)  \\
&+ \mu\lt( \B\la -\f{1}{4} + C_1( |a-1/2| +\nu(t) ) , (\om^{(4)})^2 \psi \B\ra +  C_1\la \om^2, \vp \ra \rt) \\
&+C (1 + \nu(t)) E^3(t) + C  (  | a - \f{1}{2} | + \nu(t)   )  E(t)  \\
\eal
\eeq
for some universal constant $C$. Recall $\mu = \f{1}{8C_1}$.  We have
\[
\bal
& - \f{1}{2} +C( |a-\f{1}{2}| +  \nu(t) )   + \mu C_4  
 \leq - \f{1}{4}  + C ( |a- \f{1}{2}| +  \nu(t) ) .\\
\eal
\]
Therefore, we can simplify \eqref{eq:vis_boot1} as
\beq\label{eq:vis_boot2}
\bal
\f{1}{2} \f{d}{dt} E^2(t) & \leq \lt(-\f{1}{4} + C_2(  |a-1/2| +  \nu(t) ) \rt)
( \la \om^2, \vp \ra  + \mu \la (\om^{(4)})^2, \psi \ra ) +  \\
&+C_2 (1 + \nu(t)) E^3(t) + C_2( | a - \f{1}{2} | + \nu(t)  )  E(t) , 
\eal
\eeq
where $C_2 > 0$ is some absolute constant. From the definition of $c_l, c_{\om}$ in \eqref{eq:vis_appro}, \eqref{eq:normal_vis}, we obtain 
\beq\label{eq:vis_boot3}
\bal
\B|c_l(t) + \bar{c}_l - \f{1}{3} \B| &\leq C_3 ( |a-1/2| + E(t)), \\
\B| c_{\om}(t) + \bar{c}_{\om} + 1 \B|  &\leq C_3 ( |a- 1/2| + \nu(t) +(1 + \nu(t)) E(t) ) ,\\
2 (c_l(t) + \bar{c}_l) + c_{\om}(t) + \bar{c}_{\om} &\leq - \f{1}{3} + C_3  ( |a-1/2| + \nu(t) + E(t) ),
\eal
\eeq
for some universal constant $C_3$. Recall the definition of $\nu(t)$
\[
\nu(t)  = \exp\lt( \int_0^t 2 (c_l(t) + \bar{c}_l) + c_{\om}(t) + \bar{c}_{\om}   \rt) C_l(0)^{-2} C_{\om}(0)\nu.
\]
From \eqref{eq:vis_boot3}, we see that the integrand in the exponent is bounded above by $-1/3$ and a small term. Since $\nu \leq 1$, if we choose $C_{\om}(0) = 1$ and $C_l(0)$ to be sufficiently large, then $\nu(t)$ remains small. We choose $\d ,\nu_0 >0$ such that
\beq\label{eq:vis_boot4}
( 1 + C_2 + C_3 + C_2 C_3 + C_2^2) ( \d +  \nu_0 ) < \f{1}{1000} .
\eeq
For any diffusion coefficient $0 \leq \nu \leq 1$, we choose $C_{\om}(0) = 1, C_l(0)$ large enough such that $\nu(0) \leq C_l(0)^{-2} \leq \nu_0$.
For any $ \nu \leq 1$, $ |a - 1/2| \leq \d$, it is easy to check that the bootstrap assumption
\beq\label{eq:vis_boot5}
\bal
 E(t) & \leq 8 C_2 ( |a - 1/2| +  \nu(0) ) \teq E_0 ,\\
 \nu(t) &\leq \nu(0)  \leq \nu_0, 
\eal
\eeq
can be continued provided that $E(0) \leq E_0$. In fact, if \eqref{eq:vis_boot5} holds for $t \in [0, T)$, \eqref{eq:vis_boot3} and \eqref{eq:vis_boot4} give
\[
2 (c_l(t) + \bar{c}_l) + c_{\om}(t) + \bar{c}_{\om} < -\f{1}{3} + \f{1}{20} < -\f{1}{4} ,
\]
which implies 
\beq\label{eq:vis_boot6}
\bal
\nu(t) \leq \exp\lt( -\int_0^t \f{1}{4} dt \rt) \nu(0)  \leq 
\exp( -\f{t}{4}) \nu(0) \leq \nu(0).
\eal
\eeq
When $E(t) = E_0$, the right hand side of \eqref{eq:vis_boot2} is negative. Hence \eqref{eq:vis_boot5} holds for all $t$. As a result, \eqref{eq:vis_boot3}, \eqref{eq:vis_boot4} and the bootstrap result \eqref{eq:vis_boot5} imply
\[
c_l(t) + \bar{c}_l  > \f{1}{3} - \f{1}{20} > 0, \  c_{\om}(t) + \bar{c}_{\om} < -1 + \f{1}{20} < 0.
\]
Therefore, for small odd perturbation $\om$ with $E(0) < E_0$ (we fix $C_{\om}(0)=1$ and choose $C_l(0)$ large), the bootstrap result and
$ c_{\om}(t) + \bar{c}_{\om}  < -\f{1}{2}$ imply finite time blowup in the original physical space after rescaling the time variable. We remark that $E(0) < +\infty$ implies $\om_x(0,0) =0$ due to the singular weight $\vp$.

\begin{remark}
 \eqref{eq:vis_boot6}  implies that the diffusion term in \eqref{eq:DGdy_vis} vanishes as $t \to +\infty$.
\end{remark}

\subsection{Convergence to the self-similar solution}
We use an argument similar to that in \cite{chen2019finite} to establish convergence. Instead of rewriting the estimates, we only sketch some steps.

We choose odd perturbation $\om_0$ with $E(0)< E_0$. According to the nonlinear stability analysis and \eqref{eq:vis_boot5} in the previous Section, we have a-priori estimate $E(t) < E_0$ for all $t$, where $E(t)$ is defined in \eqref{eg}.

 We choose weight $\vp$ defined in \eqref{eq:wg_vis} for the $L^2$ estimate and 
\[
 \rho  \teq  - \f{x}{ \bar{\om} } =  \f{(x^2 + b^2)^2}{2b}
 \]
 for the $H^3$ estimate. Taking time derivative on the both sides of \eqref{eq:vislin} yields 
\beq\label{eq:t}
\om_{tt} = \cL( \om_t , \nu(t)) + \cL( \om,  \nu(t)_t) + \pa_t N(\om) + \pa_t F(\bar{\om}, t) .
\eeq
A key observation is that the right hand side is linear in $\om_t$. $F(\bar{\om}, t)$ is time-dependent and we have
\beq\label{eq:terror}
\pa_t F(\bar{\om}, t) =  \dot{\nu}(t) \bar{\om}_{xx}.
\eeq
From \eqref{eq:vis_boot6}, we know that $\nu(t)$ decays exponentially fast, so does $\pa_t F(\bar{\om}, t)$. 

In \eqref{eq:t}, we have some terms related to $\dot{\nu}(t) \teq \f{d}{dt} \nu(t)$. From \eqref{eq:nu}, we have
\[
\dot{\nu}(t) = ( 2 (c_l(t) + \bar{c}_l) + c_{\om}(t) + \bar{c}_{\om} )  \nu(t) .
\] 
Using \eqref{eq:vis_boot3}, the bootstrap result \eqref{eq:vis_boot5} and the smallness of  $|a-1/2| , \nu(0) , E(t) $, we obtain 
\beq\label{eq:vis_cvg1}
  |  \dot {\nu}(t) | \les  \nu(t).
\eeq

Using an argument similar to that in the $L^2, H^4$ estimates of $\om$ and the a-priori estimate \eqref{eq:vis_boot5}, one can perform the weighted $L^2$ and $H^3$ estimates for $\om_t$ and obtain an estimate similar to that in \eqref{eq:vis_boot2}
\[
\bal
 \f{1}{2} \f{d}{dt}G^2(t)  &
 \leq - \lt(\f{1}{4} - C_4( |a-1/2| + \nu(t)  )  \rt) G^2(t)  
+ C_4 (1 + \nu(t)) E(t)  G(t)^2   \\
&+ C_4 \nu(t)  ( E(t) G(t) + E^2(t) ) + C_4 \nu(t)G(t)  
= I + II + III + IV,
\eal
\]
where $C_4>0$ is absolute and $G(t)$ is the energy for $\om_t$ 
\[
 G^2(t) \teq  \la \om_t^2, \vp \ra +  \lam \la (\om^{(3)}_t)^2, \rho \ra.
\]
$I$ comes from the estimate of linear term $\cL(\om_t, \nu(t))$, $II$ corresponds to the estimate of nonlinear term, $III$ comes from the terms related to $\f{d}{dt} \nu(t)$ and we have used \eqref{eq:vis_cvg1} to control $\dot{\nu}(t)$ and $IV$ comes from the error term $\pa_t F(\bar{\om}, t)$ (see \eqref{eq:terror}).

By choosing $|a-1/2| < \d , \nu(0) \leq \nu_0$ sufficiently small (see the discussion in the paragraph after \eqref{eq:vis_boot4}), then using a-priori estimate \eqref{eq:vis_boot5} and the smallness of $E(t)$, we further obtain
\beq\label{eq:G}
\f{1}{2} \f{d}{dt}G^2(t)
\leq -  \f{1}{5}  G^2(t) + C_4 \nu(t) ( G(t) + 1)  \leq -\f{1}{6} G^2(t) + C_5 \nu(t)
\eeq
for some universal constant $C_5$. 

\begin{remark}
Compared to the differential inequality obtained in \cite{chen2019finite}, \eqref{eq:G} contains an extra term $C_5\nu(t)$ because the error term $F(\bar{\om}, t)$ is time-dependent (see \eqref{eq:terror}).
\end{remark}

Using $\nu(t) \leq \nu(0) \exp(- \f{t}{4})$ in \eqref{eq:vis_boot6} and then solving the above ODE inequality obtain $\exp(\f{t}{3}) G^2(t)  \leq C(\nu(0), G(0)) \exp(\f{t}{12})$, or equivalently,
\[
G^2(t) \leq  C(G(0), \nu(0)) \exp(- \f{t}{4}).
\]
The remaining steps are standard. Hence, we have exponential convergence to the self-similar profile in the dynamic rescaling equation provided that $|a-1/2|, \nu(0)$ are small enough. Since $\nu(t)$ converges to $0$, such profile is the same as the inviscid profile associated with $a$. This concludes the proof of Theorem \ref{thm:blowup_vis}.

\section{Global Well-Posedness and a Blowup Criterion for $a > -1$}\label{sec:GWP}
In this Section, we consider \eqref{eq:DG} with $a > -1, \nu >0$ and prove Theorem \ref{thm:GWP}.
Without loss of generality, we assume $\nu = 1$. Then \eqref{eq:DG} becomes 
\beq\label{eq:visDG}
\bal
\om_t + a u \om_x &= u_x \om -  \CL \om , \quad u_x = H \om .
\eal
\eeq

Recall the dissipative operator $\CL\om$ defined in Section \ref{sec:Lass}
\beq\label{eq:visop}
\CL \om(x) \teq P.V.\int  \f{ \om(x) - \om(x-y) }{ m(|y|)  |y|^2 } dy,
\eeq
where $m:  [0, + \infty)  \to [0, + \infty)$ satisfies the following assumption

\quad (a) $m$ is a non-decreasing function.

\quad (b) Slightly subcritical dissipation:
 \beq\label{eq:dispm}
\lim_{r \to 0+}  r m(r)^{1/2} \int_r^1 \f{1}{s^2 m(s)} ds = +\infty , \quad \lim_{r \to 0^+} m(r) = 0.
\eeq

Recall three classes of initial data with odd symmetry or satisfying some sign property. 

\quad Class 1. $\om_0$ has a fixed sign for all $x \in \R$,

\quad Class 2. $\om_0$ is odd and $\om \geq0 $ for $x> 0$,

\quad Class 3. $\om_0$ is odd and $\om \leq 0$ for $x >0$.

It is not difficult to verify that the above symmetry and the sign property in each class are preserved by \eqref{eq:visDG} during the evolution.

\begin{lem}\label{lem:vis_l1}
(a) For $\om_0$ in class 1 and 2, $\om(t, \cdot)$ satisfies
\beq\label{eq:vis_l11}
|| \om(t)||_{L^1} \leq || \om_0||_{L^1}.
\eeq
(b) For $\om_0$ in class 3, $\om(t, \cdot)$ satisfies
\beq\label{eq:vis_l12}
|| \om(t)||_{L^1} \lesssim \exp\lt( (1+a) \int_0^t u_x(s,0) ds \rt) \cdot  || \om_0||_{L^1}.
\eeq
\end{lem}

\begin{proof}
For $\om_0$ in class 1, integrating \eqref{eq:visDG} on $\R$ and then using integration by parts yield
\[
\f{d}{dt} \int_{\R} \om dx = (1+a) \int_{\R} u_x \om -  \int_{\R} \cL(\om) dx = 0,
\]
where the last equality can be easily verified by symmetrizing the integral in the Hilbert transform and the dissipative operator $\cL$ \eqref{eq:visop}. Since $\om(t, \cdot)$ has a fixed sign, $|| \om||_{L^1}$ is conserved.

For $\om_0$ in class 2 and 3, integrating \eqref{eq:visDG} from $0$ to $\infty$ and then using integration by parts give
\beq\label{eq:tl1}
\bal
\f{d}{dt }\int_0^{\infty} \om dy
 = \int_0^{\infty} (1 + a) u_x \om dy - \int_0^{\infty} \cL(\om) dx 
 \teq I + II  .
\eal
\eeq
For $I$, since $(x-y)^{-1}$ is antisymmetric, we have
\[
\bal
I &= \f{1+a}{\pi} \int_0^{\infty} \int_0^{\infty} \lt( \f{1}{x- y} - \f{1}{x+y}\rt) \om(y)  \om(x) dx dy  \\
&= -\f{1+a}{\pi} \int_0^{\infty} \int_0^{\infty}
\f{ \om(x) \om(y)  }{x+y} dx dy  
= - \f{2(1+a)}{\pi} \int_0^{\infty} \om(x) \int_x^{\infty} \f{\om(y)}{x+y} dy.
\eal
\]
For $II$, symmetrizing the integral, we get 
\[
\bal
II & = -
\int_0^{\infty}\int_0^{\infty} \f{ \om(x) - \om(y)} { m( |x-y| ) |x-y|^2  } + \f{ \om(x) + \om(y)} { m( |x + y| ) |x + y|^2  }  d x d y   =
 -\int_0^{\infty}\int_0^{\infty}  \f{ \om(x) + \om(y)} { m( |x + y| ) |x + y|^2  }  d x d y . 
\eal
\]

For $\om_0$ in class 2, we have $\om(x), \om(y) \geq 0$ for $x, y \geq 0$ and then both $I$ and $II$ are non-positive. Hence, the time derivative in \eqref{eq:tl1} is non-positive and we prove
\[
 0 \leq \int_0^{\infty} \om(y) dy \leq \int_0^{\infty}\om_0 dy,
\]
which implies \eqref{eq:vis_l11}.

For $\om_0$ in class 3, we have $\om(x), \om(y) \leq 0$ for $x, y \geq 0$. It follows $II \geq 0$
and
\[
-\f{2}{\pi}\int_x^{\infty} \f{\om(t, y)}{x+y} dy
\leq -\f{2}{\pi} \int_0^{\infty} \f{\om(t, y)}{y} dy  = u_x(t, 0)  .
\]
Plugging the above estimates in \eqref{eq:tl1} and using $\om(x) \leq 0$ for $x \geq 0$, we prove
\[
\f{d}{dt}\int_0^{\infty} \om dy  \geq I = - \f{2(1+a)}{\pi} \int_0^{\infty} \om(x) \int_x^{\infty} \f{\om(y)}{x+y} dy \geq (1+a)u_x(0) \int_0^{\infty} \om(x) dx .
\]
Using the Gronwall inequality yields
\[
 0\leq \int_0^{\infty} ( -\om(t, y) ) dy \leq  \exp\lt( (1+a) \int_0^{\infty} u_x(s,0) ds\rt) \int_0^{\infty} ( -\om_0(y) ) dy.
\]
Note that $ || \om||_{L^1} = -2\int_0^{\infty} \om(y) dy$. We conclude \eqref{eq:vis_l12}.
\end{proof}

For the subcritical case, e.g. $\CL = \Lambda^{\al},\al > 1$, the proof of Theorem \ref{thm:GWP} follows from the standard energy estimate. For the case where $\CL$ is slightly stronger than $\Lambda$, we use the nonlinear maximum principles \cite{constantin2012nonlinear}.

\begin{proof}[Proof of Theorem \ref{thm:GWP}]
According to Lemma \ref{lem:vis_l1}, it suffices to prove that the solution exists as long as $|| \om||_{L^1}$ remains bounded. Based on the BKM type blowup criterion, see for instance \cite{beale1984remarks}, $ ||u_x||_{\infty}$ can be bounded by $|| \om||_{\infty}$ and a Sobolev extrapolation inequality with logarithmic correction. Hence, it suffices to control $|| \om||_{\infty}$ assuming that $||\om||_{L^1}$ remains bounded. Other steps are standard and hence omitted.

From the definition of the dissipative operator $\CL$, we have
\[
\bal
\om(x) \CL \om(x)  &= \int \f{ (\om(x) - \om(y)) \om(x) }{ |x-y|^2 m(|x-y|) } dy
= \f{1}{2} \int \f{ (\om(x) - \om(y))^2  }{ |x-y|^2 m(|x-y|) } dy   \\
&+ \f{1}{2} \int \f{ \om^2(x) - \om^2(y)}{ |x-y|^2 m(|x-y|) } dy   \teq  \f{1}{2} D(\om) + \f{1}{2} \CL(\om^2) ,
\eal
\]
where
\[
D(\om) =\int \f{ (\om(x) - \om(y))^2  }{ |x-y|^2 m(|x-y|) } dy .
\]
Multiplying $\om(x)$ on both sides of \eqref{eq:visDG} gives
\beq\label{eq:gwp_pf1}
\f{1}{2} \lt( \pa_t + a u \pa_x + \CL  \rt) \om^2 + \f{1}{2} D(\om) =  u_x(x) \om^2(x) .
\eeq
For some $\d >0$ to be determined, we have
\beq\label{eq:gwp_pf2}
\bal
&u_x(x)\om^2(x) = \f{\om^2(x)}{\pi} \int \f{\om(y)}{x- y} dy \\
=& \f{\om^2(x)}{\pi} \int_{|x-y| < \d} \f{ \om(y) -\om(x)}{ (x- y) }  dy + \f{\om^2(x)}{\pi}\int_{ |x-y| > \d} \f{\om(y)}{x-y}  dy \\
\leq &  \om^2(x)\lt( \int_{|x-y| < \d} \lt( \f{ \om(y) -\om(x)}{ (x- y) m(||x-y)^{1/2}} \rt)^2 dy\rt)^{1/2} \lt(\int_{|y|\leq \d} m(|y|) dy\rt)^{1/2}  + \f{\om^2(x)}{\d} ||\om||_{L^1}  \\
  \leq &\om^2(x) (D(\om))^{1/2} (\d m(\d))^{1/2} + \f{\om^2(x)}{\d} ||\om||_{L^1}.
  \leq \f{D(\om)}{4} + \d m(\d) \om^4(x) + \f{\om^2(x)}{\d} ||\om||_{L^1}.
 \eal
\eeq
We choose $\d(x)$ as follows
\beq\label{eq:vis_del}
\d^2 m(\d) = \f{ || \om||_{L^1} }{ 1 + \om^2(x) } , \ \textrm{  or equivalently  } \ \f{1}{\d} =  \f{ m(\d)^{1/2} ( \om(x)^2 + 1)^{1/2} }{ || \om||^{1/2}_{L^1}} .
\eeq
Since $m(r)$ is increasing, there exists a unique $\d$. Then \eqref{eq:gwp_pf2} can be reduced to
\beq\label{eq:gwp_pf3}
\bal
u_x(x) \om^2(x) &\leq \f{D(\om)}{4} + m(\d)^{1/2} \om^4(x)  \f{ ||\om||^{1/2}_{L^1} }{ ( \om^2(x) +1 )^{1/2}}
+ \f{ m(\d)^{1/2} ( \om(x)^2 + 1)^{1/2} }{ || \om||^{1/2}_{L^1}} ||\om||_{L^1} \om^2(x) \\
&\leq \f{D(\om)}{4} + 2 m(\d)^{1/2} ||\om||_{L^1} ( \om^2(x) + 1)^{3/2} .
\eal
\eeq
Next, we estimate a lower bound for $D(\om)$. We split the integral in $D(\om)$ into two parts
\beq\label{eq:gwp_pf41}
\bal
D(\om) &\geq \int_{ \d \leq |x-y| \leq  1} \f{ (\om(x) - \om(y))^2}{ |x-y|^2m(|x-y|) } dy \\
 &\geq \int_{ \d \leq |x - y| \leq 1} \f{ \om(x)^2}{ |x-y|^2m(|x-y|) } dy
-2 \om(x) \int_{|x-y | \geq \d }  \f{ \om(y)} {|x-y|^2m(|x-y| }dy    \\
&\geq 2 \om(x)^2 \int_{ \d \leq r \leq 1} \f{1}{r^2 m(r)} dr - \f{2| \om(x) | }{\d^2 m(\d) }  || \om||_{L^1}
\\
&= 2 \om(x)^2 \f{1}{ \d m(\d)^{1/2}} \tau(\d) - \f{2| \om(x) | }{\d^2 m(\d) }  || \om||_{L^1},
\eal
\eeq
where $\tau(r)$ is defined below
\[
\tau(r) \teq r m(r)^{1/2} \int_r^1 \f{1}{s^2 m(s)} ds.
\]
Using the definition of $\d$ in \eqref{eq:vis_del}, we yield
\beq\label{eq:gwp_pf4}
D(\om) \geq 2\tau(\d)\om(x)^2 \f{ (1 + \om^2(x))^{1/2} }{ || \om||_{L^1}^{1/2}} - 2|\om(x) |(1 + \om(x)^2 )
\geq 2 \tau(\d)\f{ |\om(x)|^3}{||\om||^{1/2}_{L^1}} - 2 (1+\om(x)^2)^{3/2}.
\eeq
Combining all the estimate \eqref{eq:gwp_pf1}, \eqref{eq:gwp_pf3} and \eqref{eq:gwp_pf4}, we derive
\[
\bal
&\f{1}{2} \lt( \pa_t + a u\pa_x + \CL  \rt) \om^2 + \f{D(\om)}{8}
\leq 2 m(\d)^{1/2} ||\om||_{L^1} ( \om^2(x) + 1)^{3/2}  - \f{D(\om)}{8} \\
 \leq& \lt(2 m(\d)^{1/2}||\om||_{L^1} + \f{1}{4}\rt)  ( \om^2(x) + 1)^{3/2}
- \f{1}{4} \tau(\d) \f{ |\om(x)|^{3} }{|| \om||_{L^1}}.
\eal
\]
At the maximal point $x_0 = \arg \max |\om|$, we get
\beq\label{eq:gwp_pf5}
\f{1}{2} \pa_t \om^2 \leq  \lt(2 m(\d)^{1/2}||\om||_{L^1} + \f{1}{4}\rt)  ( \om^2(x) + 1)^{3/2}
- \f{1}{4} \tau(\d) \f{ |\om(x)|^{3} }{|| \om||_{L^1}} .
\eeq

Let
\[
A(t) \teq  \max\lt( \exp\lt( (1+a)\int_0^t u_x(0) dt \rt), 1\rt).
\]
From Lemma \ref{lem:vis_l1}, we have $|| \om||_{L^1} \leq  A(t) || \om_0||_{L^1} $. From \eqref{eq:vis_del}, we have
\[
\d^2 m(\d) =\f{ ||\om||_{L^1}}{  \om^2(x) + 1} \leq \f{ A(t) || \om_0||_{L^1}}{  \om^2(x) + 1}.
\]
Therefore, if $\max  |\om(x) |\to +\infty $  , we get $\d \to 0$. According to the assumption \eqref{eq:dispm},
\[
m(\d)^{1/2} \to 0,\quad \tau(\d) \to +\infty.
\]
Therefore, the right hand side of \eqref{eq:gwp_pf5} is negative for sufficiently large $\om(x)$
\[
| \om(t,x) | > C( || \om_0 ||_{L^1},  m , a, \int_0^{t} u_x(s, 0) ds).
\]
It follows that  
$\sup_{ 0 \leq s < t} ||\om(s, \cdot) ||_{L^{\infty}}$ 
is bounded if $\sup_{0\leq s < t} ||\om||_{L^1} < + \infty$. 

 Using arguments similar to \cite{cordoba2004maximum} or symmetrizing the integral (see the proof of Proposition \ref{prop:lap}), we can obtain that $\cL$ is positive on $L^p$ estimates. In particular, it is positive in the $L^2$ estimate of $\om_x$. Therefore, passing from the $L^{\infty}$ estimate of $\om(t,\cdot)$ to the $H^1$ estimate follows from the standard argument, which can yield the well-posedness result. We omit it.
\end{proof}

\section{Global Well-Posedness for $a \leq -1$}\label{sec:vis_GWP}
In this Section, we prove Theorem \ref{thm:GWPc}. We consider  
\beq\label{eq:DG_vis2}
\bal
\om_t + a u \om_x &= u_x \om -  \Lam^{\g} \om  , \quad u_x = H \om ,
 \eal
\eeq
with $a \leq -1$, where $\g \in (0,2 ]$. 

Denote $p = |a| \geq 1$. Performing the $L^p$ estimate of \eqref{eq:DG_vis2} and using integration by parts yield
\beq\label{eq:ap}
\bal
\f{1}{p} \f{d}{dt} || \om ||^p_{L^p} 
 &= \int_{\R}  u_x |\om|^p  - \f{a}{p} u (|\om|^p)_x   dx - \int_R 
\om | \om|^{p-2} \Lam^{\g}\om  dx \\
& = \int_{\R} (1 + \f{a}{p}) u_x |\om|^p dx - \int_R 
\om | \om|^{p-2} \Lam^{\g}\om  dx  \leq 0,
\eal
\eeq
where we have used $1 + \f{a}{p} =0$ and Proposition \ref{prop:lap} below about the fractional Laplacian.

We mainly focus on the proof of the global well-posedness with critical or supercritical dissipation. The global well-posedness in the subcritical case, i.e. $\g > |a|^{-1} $ follows from the standard Sobolev energy estimates.

For $a<-1$, it seems difficult to apply the method of modulus of continuity \cite{kiselev2007global} or the nonlinear maximal principle \cite{constantin2012nonlinear} to establish the global well-posedness due to the vortex stretch term in \eqref{eq:DG_vis2}. 

Our proof of the global well-posedness is based on a-priori $L^{|a| + \e}$ estimate for sufficiently small $\e >0$ by exploiting the cancellation between the transport term and the vortex stretching term. Firstly, we need a positivity estimate about the fractional Laplacian, which has been used to derive \eqref{eq:ap}.

\begin{prop}\label{prop:lap}
For $\g \in [0,2], p  \geq 1 $, we have 
\[
\int_R \om | \om|^{p-2}  \Lam^{\g} \om d x \geq   C \min( p-1, \f{1}{p})   || \Lam^{\g/2} |\om|^{p/2} ||^2_{L^2}.
\]
\end{prop}
Similar results on the torus $T^d$ when $p \geq 2$ is even have been established in e.g. 
\cite{constantin2015long}, \cite{cordoba2004maximum}. We defer the proof to Appendix \ref{app:lap}.

We separate the proof of Theorem \ref{thm:GWPc} for $a= -1$ and $a < -1$.

\subsection{Global Well-Posedness for $a  < -1$}
\begin{proof}
Denote $ p = |a| < 1$. From \eqref{eq:ap}, we know a-priori that $ \om(t, \cdot)$ is $L^p$ bounded. 

Let $ 0 < \e \leq 1$, $ (2p)^{-1} \leq \g_1 < p^{-1}$ to be determined and $q = p+\e$. We consider $\g \in [\g_1, p^{-1}]$.
Multiplying $\om |\om|^{q-2}$ on both sides of \eqref{eq:DG_vis2} and then performing $L^{q}$ estimate yield
\[
\bal
\f{1}{q} \f{d}{dt} || \om||^q_{L^q} 
&= \la u_x , | \om|^q  \ra 
-  \f{a}{q} \la u, ( |\om|^q )_x \ra  - \la \om |\om|^{q-2} , \Lam^{\g} \om \ra \\
& = \f{ \e }{q}  \la u_x , |\om|^q \ra -  \la \om |\om|^{q-2} , \Lam^{\g} \om \ra 
\teq I -  II, 
\eal
\]
where we have used integration by parts and $1 + \f{a}{q} = \f{q + a}{q} = \f{\e}{q}$ to obtain the second equality. The small parameter $\e$ is from the cancellation between the transport term and the vortex stretch term. Applying Proposition \ref{prop:lap} yields 
\[
II \geq C \min(q-1 , q^{-1}) || \Lam^{\g/2} |\om|^{q/2} ||^2_{L^2}
\geq C \min(p-1 , p^{-1}) || \Lam^{\g/2} |\om|^{q/2} ||^2_{L^2},
\]
where we have used $q = p+ \e \in [p, p+1]$. Applying the Sobolev embedding $\dot{H}^{\g/2} \hookrightarrow L^{\f{2}{1-\g}}$, we obtain 
\[
II \geq c_p ||  |\om|^{q/2} ||^2_{\dot{H}^{\g/2}} \geq c_p || |\om|^{q/2} ||^2_{L^{ \f{2}{1-\g}  } }
= c_p || \ |\om|^q ||_{L^{ \f{1}{1-\g}} },
\]
where we have used $ (2p)^{-1} \leq \g_1 \leq \g \leq p^{-1}$ so that the constant in the Sobolev embedding (the second inequality) only depends on $p$. Using the H\"older inequality, $L^{ \f{1}{\g}}$ boundedness of the Hilbert transform and $ \g^{-1} \in [p, 2p]$, we yield
\[
I 
\leq \f{\e}{q} || u_x||_{L^{ \f{1}{\g} }} || \ | \om|^q ||_{L^{\f{1}{1-\g}}}
\leq \e C_p || \om ||_{L^{ \f{1}{\g} }}  || \ | \om|^q ||_{L^{\f{1}{1-\g}}},
\]
where we have used $q^{-1} \leq C_p$. 

Combining the above estimates, we prove 
\beq\label{eq:GWP_boot}
\f{1}{q} \f{d}{dt} || \om||^q_{L^q}
\leq (  \e C_p  || \om ||_{L^{\f{1}{\g}}} - c_p ) || \ | \om|^q ||_{L^{\f{1}{1-\g}}}.
\eeq

Next, we determine the parameter $\g_1$ and $\e$. Notice that for any $m \in [p , \infty]$, the H\"older inequality and the simple bound $|| \om||_{L^{\infty}} \leq || \om ||_{H^1}$ imply
\beq\label{eq:interp}
|| \om ||_{L^m} \leq || \om ||_{L^p} + || \om||_{L^{\infty}}
\leq || \om ||_{L^2} + || \om||_{ H^1} .
\eeq
For any initial data $\om_0 \in L^{p} \cap H^1$, we choose $\e, \g_1$ as follows 
\[
\e = \min(1,  \f{c_p} { 4 C_p (1 + || \om_0 ||_{L^p} + || \om_0 ||_{H^1} ) } ) , \quad
\g_1 = \max( (2p)^{-1},  \f{1}{ p + \f{\e}{2}} ).
\]
Clearly, we have $ \e \in (0, 1],  \g_1 \in [ (2p)^{-1}, p^{-1} )$. For this $\g_1$ and any $ \g$ in $[\g_1, p^{-1}]$, we show that \eqref{eq:DG_vis2} with initial data $\om_0$ is globally well-posed. Firstly, we have 
\beq\label{eq:subcrit}
  p \leq \g^{-1} \leq \g_1^{-1} \leq p + \f{\e}{2} < q.
\eeq

Consider the bootstrap assumption $|| \om(t, \cdot) ||_{L^q} \leq || \om_0||_{L^q} $. Under the bootstrap assumption, using a-priori $L^p$ boundedness of $|| \om||_{L^p}$, a simple interpolation, \eqref{eq:interp} and the definition of $\e$, we have
\[
\bal
\e C_p || \om ||_{L^{\f{1}{\g}}}
&\leq \e C_p (  || \om||_{L^p}  + || \om ||_{L^q} )
\leq \e C_p (  || \om_0||_{L^p}  + || \om_0 ||_{L^q} ) \\
& \leq 2 \e C_p  ( || \om_0 ||_{L^p} + || \om_0||_{H^1} ) \leq \f{c_p}{2}.
\eal
\]

Combining \eqref{eq:GWP_boot}, we know that the bootstrap assumption can be continued. Hence, we obtain a-priori $L^q$ boundedness of $\om(t, \cdot)$. 

From \eqref{eq:subcrit}, we have $q \g > 1$. Therefore, this a-priori $L^q$ estimate is subcritical. 
The proof of the global well-posedness follows from the standard Sobolev energy estimates. We omit it.
\end{proof}

\subsection{Global Well-Posedness for $a  = -1 $}\label{sec:CCF}
We focus on the critical case of \eqref{eq:DG_vis2}, i.e. $\g = 1$. The proof is based on establishing $L^{1+\e}$ a-priori estimate.

\begin{proof}
Let $ 0 < \e \leq 1/4 $ to be determined and $q = 1 + \e$. From \eqref{eq:ap}, we have a-priori $L^1$ estimate of $\om(t,\cdot)$. Multiplying $\om |\om|^{q-2}$ on both sides of \eqref{eq:DG_vis2} ($a=-1, \g=1$) and then performing $L^q$ estimates, we yield 
\beq\label{eq:bootCCF}
\bal
\f{1}{q} \f{d}{dt} || \om||^q_{L^q} 
&= \la u_x , | \om|^q  \ra 
+  \f{1}{q} \la u, ( |\om|^q )_x \ra  - \la \om |\om|^{q-2} , \Lam^{\g} \om \ra \\
& = \f{ \e }{ q}  \la u_x , |\om|^q \ra -  \la \om |\om|^{q-2} , \Lam \om \ra 
\teq I -  D.
\eal
\eeq

For $I$, we cannot simply apply the H\"older inequality since when we bound 
$||u_x||_{L^m}$ by $||\om||_{L^m}$ 
with $m$ sufficiently close to $1$, we will have a large constant, which compensates the small parameter $\e$. We exploit the cancellation between $u_x$ and $|\om|^q$ for $q$ sufficiently close to $1$. Symmetrizing the integral, we get 
\[
\bal
\la u_x, |\om|^q \ra
&= \f{1}{\pi} P.V. \iint  \f{\om(y)}{x- y}  |\om(x)|^q dy dx
= \f{1}{2\pi} \iint \om(x) \om(y) \f{ \om(x) |\om|^{q-2} - \om(y) |\om|^{q-2} }{x- y} dx dy , \\
D & =  \la \om |\om|^{q-2} , \Lam \om \ra 
= C \cdot  P.V.\iint  \f{\om(x)  - \om(y)}{ |x-y|^2} \om(x) |\om(x)|^{q-2} dy dx  \\
& = C \iint \f{\om(x)  - \om(y)}{ |x-y|^2} ( \om(x) |\om(x)|^{q-2} - \om(y)  |\om(y)|^{q-2} ) dy dx ,
\eal
\]
where $C > 0$ is absolute. Next, we show that for any $X, Y $
\beq\label{eq:XY2}
(X- Y) ( X |X|^{q-2} - Y |Y|^{q-2}) 
\geq  ( X |X|^{q-2} - Y |Y|^{q-2})^2  |XY|^{\f{2-q}{2}}.
\eeq
Without loss of generality, we assume that $ |X | \geq |Y|$ and $X \geq  0$. Then we have 
$ X |X|^{q-2} - Y |Y|^{q-2} \geq 0 $ and it suffices to verify 
\[
S \teq  ( X - Y)  - 
( X |X|^{q-2} - Y |Y|^{q-2})  |XY|^{\f{2-q}{2}} \geq 0.
\]
It follows from the following identity 
\[
S = (|X|^{1-q/2} - |Y|^{1-q/2}) ( X |X|^{q/2- 1} + Y |Y|^{q/2 - 1}  ) \geq 0,
\]
where we have used $|X| \geq |Y|, X \geq 0 $ and that both terms are non-negative. 

Applying \eqref{eq:XY2} with $X = \om(x), Y= \om(y)$, we yield 
\[ 
D \geq C \iint \f{ ( \om(x) |\om(x)|^{q-2} - \om(y)  |\om(y)|^{q-2} )^2 }{ |x-y|^2}
|\om(x) \om(y)|^{\f{2-q}{2}} dx d y.
\]
Hence, using the Cauchy-Schwartz inequality, we prove 
\[
|\la u_x, |\om|^q \ra|
\les D^{1/2} \lt( \iint |\om(x)\om(y) |^{2-  \f{2-q}{2}} dx dy \rt)^{1/2}
\les D^{1/2} || \om ||^{\f{2+q}{2}}_{L^{ \f{2+q}{2}}} .
\]
Applying Proposition \ref{prop:lap} with $p = q = 1 + \e$, we have 
\[
|| \ | \om |^{q/2} ||_{\dot{H}^{1/2}}^2 \les \e^{-1} D. 
\]
Using the interpolation in Lemma \ref{lem:inter}, we obtain 
\[
|| \om ||^{\f{2+q}{2}}_{L^{ \f{2+q}{2}}} \les || \om||_{L^1} || \ | \om |^{q/2} ||_{\dot{H}^{1/2}}.
\]
Combining the above estimates, we derive 
\[
|\la u_x, |\om|^q \ra| \leq C_6 || \om||_{L^1} D^{1/2} (\e^{-1} D)^{1/2}
= C_6 \e^{-1/2} D || \om||_{L^1},
\]
where $C_6>0$ is absolute. Substituting the above estimate in \eqref{eq:bootCCF}, we prove
\[
\f{1}{q} \f{d}{dt} || \om||^q_{L^q} 
\leq ( \f{\e}{q} C_6 \e^{-1/2} || \om||_{L^1} - 1 ) D
\leq (C_6 \e^{1/2 } || \om_0||_{L^1} - 1) D,
\]
where we have used $q \geq 1$ and a-priori $L^1$ estimate  $|| \om(t,\cdot )||_{L^1} \leq || \om_0||_{L^1}$. Choosing $\e$ sufficiently small, we obtain that $|| \om(t, \cdot)||_{L^q}, q= 1 + \e$ is decreasing. 

With this a-priori $L^q$ estimate, the dissipation $\Lam$ is subcritical.
The proof of the global well-posedness follows from the standard Sobolev energy estimates. We omit the details.
\end{proof}

\vspace{0.2in}
\noindent
{\bf Acknowledgments.} The author would like to thank Thomas Hou for helpful comments on earlier version of this work. This research was supported in part by the NSF Grants DMS-1907977 and DMS-1912654.

\appendix
\section{}
\subsection{Properties of the Hilbert transform}
Throughout this section, we assume that $\om$ is smooth and decays sufficiently fast. The general case can be obtained easily by approximation. We list several properties of the Hilbert transform and we refer the reader to \cite{chen2019finite} for the proof. 

\begin{lem}[The Tricomi identity]\label{lem:tric}
 We have
\beq\label{eq:tric}
H( \om H\om)  =  \f{1}{2} (  (H \om)^2 - \om^2  ).
\eeq
\end{lem}

\begin{lem}\label{lem:commute} Suppose that $u_x = H\om$ and $\f{\om}{x}$ is well-defined. Then
\beq\label{lem:vel0}
 (H\om)(x) = (H\om)(0) + x H\lt( \f{\om}{x}\rt) .
\eeq

\end{lem}

The following cancellation results and estimates are crucial to establish the linear stability in Section \ref{sec:vis_blowup}. 

\begin{lem}\label{lem:vel_a1} Suppose $u_x = H\om$.
(a)  We have
\begin{align}
\int \f{  ( u_x - u_x(0) ) \om    }{  x  } & = \f{\pi}{2} ( u_x^2(0) + \om^2(0)) \geq 0 \label{lem:vel3} .
\end{align}
Furthermore, if $\om$ is odd (so does $u_{xx}$ due to the symmetry of Hilbert transform), then
\begin{align}
\int \f{  ( u_x - u_x(0) ) \om    }{  x^3  } & =  \f{\pi}{2} ( \om_x^2(0) -u^2_{xx}(0)) =\f{\pi}{2}  \om_x^2(0) \geq 0 \label{lem:vel4}.
\end{align}
In particular,  \eqref{lem:vel3} vanishes if $ u_x(0) = \om(x) =0$; \eqref{lem:vel4} vanishes if $\om_x(0) = 0 $ .

(b)  We have
\beq\label{lem:vel5}
   \int u_{xx} \om_x x  = 0  .
\eeq

(c) Hardy inequality: Suppose that $\om$ is odd and $\om_x(0)=0$. For $p = 2, 4$,
\beq\label{eq:hd1}
\int \f{  (u  - u_x(0) x)^2} { |x|^{p+2} }  \leq  \lt(\f{2}{p +1 } \rt)^2   \int \f{  (u_x - u_x(0))^2}{ |x|^{ p } } = \lt(\f{2}{p +1 } \rt)^2  \int   \f{  \om^2}{ |x|^{ p } }.
\eeq
\end{lem}

The first inequality in \eqref{eq:hd1} is the standard Hardy inequality \cite{Har52}.

\subsection{Estimates of the linearized operator}

We introduce $S, R$ defined below to simplify the quantities in \eqref{eq:vislin_L20} and \eqref{eq:vislin_51}
\beq\label{eq:damping}
\bal
S(\bar{c}_l, \bar{c}_{\om}, a, x) & \teq
\f{1}{2\vp} (  ( \bar{c}_l x +a  \bar{u} ) \vp )_x + ( \bar{c}_{\om} + \bar{u}_x )
+ \f{1}{\vp} \lt( \f{4 b}{9}\f{1}{x^2}  + \f{a^2}{4 b }  (\bar{\om}_x \vp)^2x^4   \rt), \\
R(\bar{c}_l, \bar{c}_{\om}, a, x) & \teq  \f{1}{2 \psi} (  (\bar{c}_l x + a \bar{u})\psi)_x  - (4 \bar{c}_l + 4 a \bar{u}_x - \bar{c}_{\om}-\bar{u}_x),
\eal
\eeq
where $ b = \sqrt{ \f{3}{8}}$, $\bar{c}_l, \bar{c}_{\om}, \bar{\om}, \bar{u}_x$ are given in \eqref{eq:vis_appro} and $\vp, \psi$ is given in \eqref{eq:wg_vis}.

\begin{lem}\label{lem:vis_dp}
(a) 
$S$ satisfies the following pointwise estimate
\beq\label{eq:vis_Sneg}
\bal
S\lt(  \f{1}{3} ,  -1,  \f{1}{2}, x\rt) 
=&   \f{1}{2\vp} \lt(  (  \f{x}{3} +  \f{\bar{u} }{2}) \vp \rt)_x + ( -1 + \bar{u}_x )  + \f{1}{\vp} \lt(  \f{4 b}{9}\f{1}{x^2}  + \f{ (\bar{\om}_x \vp)^2x^4}{16  b }   \rt) \leq  -\f{1}{2} .
\eal
\eeq
(b)
For any $x$, we have
\beq\label{eq:vis_Scont}
\B|\B|  \f{ (x\vp)_x}{2\vp} \B|\B|_{\infty} \les  1, \quad
\B|\B|  \f{ (\bar{u}\vp)_x}{2\vp} \B|\B|_{\infty}  \les 1 ,\quad 
\B|\B|\f{1}{\vp}  (\bar{\om}_x \vp)^2x^4 \B|\B|_{\infty} =|| \bar{\om}_x^2 \vp x^4 ||_{L^{\infty}} \les 1.
\eeq
In particular, 
for $a$ close to $\f{1}{2}$ and some universal constant $C > 0$, we have
\beq\label{eq:vis_Sneg2}
S(\bar{c}_l, \bar{c}_{\om}, a, x) 
\leq -1/2+ C( |a-1/2| +  \nu(t) ).
\eeq
(c) For $a$ close to $\f{1}{2}$ and some universal constant $C > 0$, we have
\beq\label{eq:vis_Sneg3}
R(\bar{c}_l, \bar{c}_{\om}, a, x) \leq - \f{1}{3} + C( |a-1/2| +  \nu(t) ).
\eeq
\end{lem}

\begin{proof}[Proof of \eqref{eq:vis_Sneg}]
 Denote 
\beq\label{eq:vis_S12}
 S_1(x)  = \f{1}{2\vp} (  (  \f{x}{3} +  \f{\bar{u}}{2}  ) \vp )_x + ( -1 + \bar{u}_x )  + \f{1}{2},
 \quad
 S_2(x) =   \f{1}{\vp} \lt(  \f{4 b}{9}\f{1}{x^2}  + \f{  (\bar{\om}_x \vp)^2x^4}{16  b }   \rt)  .
\eeq

Recall the definition of $\vp, \bar{u}, \bar{\om}$ in \eqref{eq:wg_vis} and \eqref{eq:vis_appro}. A direct calculation yields 
\[
\bal
\f{ (x \vp)_x }{\vp}&=  \f{x^4}{ (x^2 + b^2)^3}  \lt( \f{(x^2 + b^2)^3}{x^3} \rt)_x 
=  \f{x^4}{ (x^2 + b^2)^3}  \lt( \f{6 x(x^2 + b^2)^2 }{ x^3} 
- \f{ 3 x(x^2 + b^2)^3}{x^4} 
\rt) \\
& =  \f{6 x^2} {x^2 + b^2} - 3  = \f{3 x^2 - 3 b^2}{x^2 + b^2}, \\
\f{ (\bar{u} \vp)_x}{\vp} & = \f{x^4}{ (x^2 + b^2)^3} \lt(\f{(x^2 + b^2)^2}{x^3}  \rt)_x 
= \f{x^4}{ (x^2 + b^2)^3} \lt(\f{4x (x^2 + b^2)}{x^3} - \f{3  (x^2 + b^2)^2}{x^4}  \rt)  \\
& = \f{4x^2}{ (x^2 + b^2)^2}  - \f{3}{x^2 + b^2} =  \f{x^2 - 3 b^2}{ (x^2 + b^2)^2}, \\
\bar{\om}^2_x \vp x^4 & = \lt( \lt( \f{-2bx}{ (b^2 + x^2)^2} \rt)_x\rt)^2 \f{(x^2+b^2)^3}{2b} 
= \lt( \f{-2b}{ (x^2 + b^2)^2}  + \f{2bx \cdot 4 x}{ (b^2 + x^2)^3} \rt)^2 \f{(x^2+b^2)^3}{2b}  \\
& = \lt( \f{ 2 b ( 3 x^2 - b^2)}{ (x^2 + b^2)^3} \rt)^2 \f{(x^2+b^2)^3}{2b}
= \f{2b(3x^2 - b^2)^2}{ (x^2 + b^2)^3},
\eal
\] 
where $ b = \sqrt{\f{3}{8}}$. Using the above calculations and a simple estimate yields \eqref{eq:vis_Scont}.

Plugging the above computations and the formula of $\bar{u}_x$ given in \eqref{eq:vis_appro} in $S_1(x)$, we get 
\[
\bal
S_1(x) &= \f{1}{6} \f{3 x^2 - 3 b^2}{x^2 + b^2} + \f{1}{4} \f{x^2 - 3 b^2}{ (x^2 + b^2)^2} + \f{b^2 - x^2}{(x^2 + b^2)^2} - \f{1}{2} 
= \f{-2 b^2}{ 2(x^2 + b^2)} + \f{ b^2 - 3 x^2}{ 4(x^2+b^2)^2}  \\
& = \f{ -4 b^2(x^2 + b^2) + b^2 - 3x^2}{ 4(x^2 + b^2)^2} = 
\f{ -\f{3}{2}(x^2 + b^2) + b^2 -3x^2}{ 4(x^2 + b^2)^2} =  \f{ -\f{1}{2} b^2 - \f{9}{2}x^2}{{ 4(x^2 + b^2)^2}},
\eal
\]
where we have used $b^2 = \f{3}{8}$. Using the calculations about $\bar{\om}_x^2 \vp x^4$, we derive
\[
S_2(x) = \f{2 b x^4}{ (x^2 + b^2)^3} \f{4b}{9 x^2 } + \f{1}{16 b} \f{2b(3x^2 - b^2)^2}{ (x^2 + b^2)^3} = \f{ 64 b^2 x^2 + 9 (3 x^2 - b^2 )^2 }{72(x^2 + b^2)^3}
= \f{ 9 b^4 + 81 x^4 + 10 b^2 x^2 }{72(x^2 + b^2)^3}.
\]
From \eqref{eq:vis_Sneg} and  \eqref{eq:vis_S12}, we have $S(\f{1}{3}, -1, \f{1}{2}, x) + \f{1}{2} = S_1(x) + S_2(x)$. To prove \eqref{eq:vis_Sneg}, it suffices to verify that $S_1(x) +S_2(x) \leq 0 $, which is equivalent to
\[
(-\f{1}{2} b^2 -\f{9}{2}x^2) \cdot 18 (x^2 + b^2) + 9 b^4 + 81 x^4 + 10 b^2 x^2 
=  ( 10 -9  - 81)	b^2 x^2 \leq 0.
\]
Hence, we yield \eqref{eq:vis_Sneg}. 

Recall the definition of $\bar{c}_{\om}, \bar{c}_l, S$ in \eqref{eq:vis_appro}, \eqref{eq:damping}. Clearly, we have 
\[
\bal
S(\bar{c}_l, \bar{c}_{\om}, a, x) & - S(\f{1}{3}, -1, \f{1}{2}, x)
= (\bar{c}_l - \f{1}{3}) \f{ (x \vp)_x}{2 \vp}  + (a- \f{1}{2}) \f{ ( \bar{u} \vp)_x}{ 2\vp}
+ (\bar{c}_{\om} + 1) + (a^2 - \f{1}{4}) \f{1}{4b} \bar{\om}^2_x \vp x^4  \\
& = - (a-\f{1}{2}) \bar{u}_x(0) \f{ (x \vp)_x}{2 \vp} 
+ (a- \f{1}{2}) \f{ ( \bar{u} \vp)_x}{ 2\vp}
 - \f{\nu(t) \bar{\om}_{xxx}(0) }{\bar{\om}_x(0)} + (a^2 - \f{1}{4}) \f{1}{4b} \bar{\om}^2_x \vp x^4 .
\eal
\]
Note that $\bar{u}_x(0), \bar{\om}_{xxx}, \bar{\om}_x(0)\neq 0$ are some absolute constants. Plugging the estimates \eqref{eq:vis_Sneg} and \eqref{eq:vis_Scont} into the above calculation yields \eqref{eq:vis_Sneg2}.

The proof of \eqref{eq:vis_Sneg3} is similar. Recall the definition of $\bar{u}, \bar{u}_x, \psi = \f{(x^2+b^2)^3}{2b}$ in \eqref{eq:vis_appro} and \eqref{eq:wg_vis}. A direct calculation yields 
\beq\label{eq:xpsi}
\bal
 \f{x \psi_x}{\psi}  &=  \f{ x \cdot 6x}{ x^2 + b^2}, \quad 
 \f{ (x\psi)_x}{ \psi } = 1 + \f{x \psi_x}{\psi} = 1 + \f{6x^2}{x^2 + b^2},   \\
  \quad \f{ (\bar{u}\psi)_x}{\psi}& = \bar{u}_x + \f{\bar{u}}{x} \f{x\psi_x}{ \psi}
 =\f{b^2 - x^2}{ (x^2 + b^2)^2}+ \f{1}{x^2 + b^2}  \f{ 6x^2}{ (x^2 +b^2)}
 = \f{ b^2 + 5x^2}{ (b^2 + x^2)^2}.
 \eal
\eeq
Taking $\bar{c}_l = \f{1}{3}, \bar{c}_{\om} =-1, a = \f{1}{2}$ in \eqref{eq:damping}, we get
\[
\bal
R(\f{1}{3}, -1, \f{1}{2}, x) &= 
  \f{1}{2 \psi} (  ( \f{1}{3} x + \f{1}{2} \bar{u})\psi)_x  - ( \f{4}{3} + 2 \bar{u}_x +1-\bar{u}_x) 
  =\f{1}{6}(1 + \f{6x^2}{x^2 + b^2}) + \f{1}{4}\f{ b^2 + 5x^2}{ (b^2 + x^2)^2}
  -\f{7}{3} - \bar{u}_x  \\
  &   =-\f{13}{6} + \f{x^2}{x^2 + b^2}  + \f{ b^2 + 5x^2}{4(x^2 + b^2)^2} - \f{b^2 - x^2}{(x^2+b^2)^2}
  = -\f{13}{6} + \f{x^2 }{(x^2 + b^2)} +\f{9x^2 -3 b^2} { 4(b^2 + x^2)^2} .
  \eal
\]
Next, we show that $R(\f{1}{3}, -1, \f{1}{2}, x) \leq -\f{1}{3}$, which is equivalent to
\[
  0 \geq -  \f{11}{6}\cdot 4 (x^2 + b^2)^2 + 4x^2(x^2+b^2) + 9x^2 -3 b^2 = 
   -( \f{22}{3} - 4 )x^4  + (9 + 4b^2 - \f{22}{3} \cdot 2 b^2)  x^2- (\f{22}{3} b^4 + 3 b^2) \teq I.
\] 
Using $b^2 = \f{3}{8}$, we can rewrite the above inequality as 
\[
I = -\f{10}{3} x^4 + (9 + 4 \cdot \f{3}{8} - \f{44}{3}\cdot \f{3}{8}) x^2 - \f{3}{8}( \f{22}{3} \cdot\f{3}{8} + 3) =  -\f{10}{3}x^4 + 5 x^2 - \f{3}{8} \cdot \f{23}{4}.
\]
Since $4 \times \f{10}{3}  \times \f{3}{8} \times \f{23}{4} - 5^2 = \f{230}{8} - 25 >0$, we get $I \geq 0 $ and therefore 
\[
R(\f{1}{3}, -1, \f{1}{2}, x) \leq -\f{1}{3}.
\]
Using this estimate and an argument similar to that in the proof of \eqref{eq:vis_Sneg2} implies \eqref{eq:vis_Sneg3}.
\end{proof}

\subsection{Inequalities about the Fractional Laplacian
}\label{app:lap}
In this Section, we prove Proposition \ref{prop:lap} and establish an interpolation Lemma \ref{lem:inter}. Firstly, we need an elementary Lemma.

\begin{lem}\label{lem:xyn}
For any $x, y \geq 0$ and $n \geq 1$, we have 
\[
 (x^{ \f{n+1}{2}} - y^{\f{n+1}{2}})^2
 \leq  C n (x^n - y^n) (x - y) .
\]
\end{lem}

\begin{proof}
The inequality is trivial if $x =0$, $y=0$ or $x = y$. For other cases, without loss of generality, we assume $x > y >0$ . Denote $\lam = \f{x}{y}  > 1$. Then the inequality is equivalent to 
\[
(\lam^{\f{n+1}{2}} - 1)^2 \les n (\lam^n - 1)(\lam - 1).
\]  
For any $m  \geq 1$, there exists $k \in Z^+$ such that $ k \leq m  < k+1$. For $\lam >1$, we have 
\[
  \lam^{m -1} + \lam^{m-2} + ...+ \lam^{m-k}
\leq  \f{ \lam^m - 1 }{\lam-1}\leq 
 \lam^{m -1} + \lam^{m-2} + ...+ \lam^{m-k} + 1  .
\]
The lower bound and the upper bound are comparable.  Suppose that $ k \leq \f{n+1}{2} < k+1$. We have $ 2k-1 \leq  n < 2k + 1$. Applying the above inequality with $m = n, \f{n+1}{2}$, we get 
\[
\f{ (\lam^n - 1) (\lam- 1)} { (\lam^{ (n+1)/2} -1)^2}
\geq  C \f{ \lam^{n-1} + \lam^{n-2} + ... + \lam^{ n - 2k + 1 } + 1  }{ (\lam^{ (n+1)/2-1}
+ \lam^{ (n+1)/2-2} + ... + \lam^{ (n+1) / 2 - k} )^2}.
\]
Using the Cauchy-Schwartz inequality, we prove 
\[
( \lam^{n-1} + \lam^{n-2} + ... + \lam^{ n - 2k + 1 } + 1  )  n 
\geq (\lam^{ (n+1)/2-1}
+ \lam^{ (n+1)/2-2} + ... + \lam^{ (n+1) / 2 - k} )^2.
\]
Combining the above two estimates completes the proof.
\end{proof}

\begin{proof}[Proof of Proposition \ref{prop:lap}]
Fro $\g = 0, 2$, the results follow from integration by parts. Recall the definition of the fractional Laplacian $\Lam^{\g}$ with $\g \in (0,2)$ (see e.g. \cite{cordoba2004maximum} and the references therein)
\[
\Lam^{\g} \om(x) = C_{\g} P.V. \int_{\R} \f{\om(x) - \om(y)} { |x- y|^{1+\g}} dy .
\]
Symmetrizing the integral, we have 
\beq\label{eq:lap_pf0}
\bal
I &\teq 
\int_{\R} \om |\om|^{ p-2} \Lam^{\g} \om dx 
= C_{\g} P.V. \int_{\R}  \om(x) |\om(x)|^{ p-2} \f{\om(x) - \om(y)} { |x- y|^{1+\g}} dy  \\
 &= \f{C_{\g}}{2} \int_{\R} 
\int_{\R}  ( \om(x) |\om(x)|^{ p-2} - \om(y) |\om(y)|^{ p-2} ) \f{\om(x) - \om(y)} { |x- y|^{1+\g}} dy . 
\eal
\eeq
Next, we show that 
\beq\label{eq:lap_pf1}
\bal
 II(x, y)  &\teq (\om(x) |\om(x)|^{ p-2} - \om(y) |\om(y)|^{ p-2} ) (\om(x) - \om(y)) \\
 &\geq C \min(p-1, p^{-1}) ( |\om(x)^{p/2} - |\om(y)|^{p/2} )^2 .
 \eal
\eeq

If $\om(x) =0$ or $\om(y) =0$, the above inequality is trivial. For $p=1$, the inequality is verified by discussing the sign of $\om(x), \om(y)$. For $p > 1$, if $sgn(x) \neq sgn(y)$, using the Cauchy-Schwartz inequality, we have
\[
II(x,y)  = ( |\om(x)|^{p-1} + |\om(y)|^{p-1}) (|\om(x)| + |\om(y)|) 
\geq ( |\om(x)|^{p/2} + |\om(y)|^{p/2})^2,
\]
which implies \eqref{eq:lap_pf1}. If $\sgn(x) = \sgn(y)$, we get 
\[
II(x, y) = ( |\om(x)|^{p-1}  - |\om(y)|^{p-1}) (|\om(x)| - |\om(y)|) 
=  (X^n - Y^n) (X-Y) ,
\]
where $n, X,Y$ are defined as follows: 
\[
X = |\om(x)|^{\al}, \quad Y = |\om(y)|^{\al}, \quad \al = \min(p-1, 1),
\quad n = \max(p-1, 1) \cdot \al^{-1} .
\]
Clearly, $n \geq 1$, $ (n+1) \al =\max(p-1, 1) + \min(p-1, 1)  = p$. Applying Lemma \ref{lem:xyn}, we yield
\[
\bal
II(x, y) &\geq C n^{-1}  ( X^{ (n+1) /2} - Y^{ (n+1)/2})^2 .
\eal
\]
Since 
\[
n^{-1} = \f{ \min(p-1, 1)}{ \max(p-1,1 )}\geq C \min(p-1, p^{-1} )  ,
\]
we prove \eqref{eq:lap_pf1} when $sgn(x) = sgn(y) \neq 0$. Combining different cases, we prove \eqref{eq:lap_pf1}. Plugging the estimate \eqref{eq:lap_pf1} in \eqref{eq:lap_pf0} conclude 
\[
\bal
I & \geq   C C_{\g} \min(p-1, p^{-1}) 
\iint  \f{ (  |\om(x)|^{ p/2} - |\om(y)|^{p/2} )^2}{ |x-y|^{1+\g}}  dx dy
=   C \min(p-1, p^{-1})  || \Lam^{\g/2} |\om|^{p/2} ||^2_{L^2} ,
\eal
\]
where $C  > 0$ is some absolute constant.
\end{proof}

We have used the following Lemma in Section \ref{sec:CCF} to prove the global well-posedness.
\begin{lem}\label{lem:inter}
Suppose that $q \in [1, 5/4]$ and $ \om \in L^1, |\om|^{q/2} \in \dot{H}^1$. We have 
\[
|| \om ||^{\f{2+q}{2}}_{L^{ \f{2+q}{2}}} \les || \om||_{L^1} || \ | \om |^{q/2} ||_{\dot{H}^{1/2}}.
\]
\end{lem}

\begin{remark}
We will only apply this result for $q$ close to $1$. It seems that the above 
Gagliardo-Nirenberg type interpolation is hard to find in the literature. We provide a proof.
\end{remark}

\begin{proof}
Denote $ f = |\om|^{q/2}$. The inequality is equivalent to 
\[
|| f ||^{1 + 2 / q}_{L^{1 + 2/q}}
\les || f ||^{2/q}_{L^{2/q} } || f||_{\dot{H}^{1/2}}. 
\]
Let $\chi : R \to [0, 1]$ be an even smooth cutoff function such that $\chi(x) = 1$ for $|x| \leq 1$ and $\chi(x) = 0$ for $|x | \geq 2$. Denote $ h = \cF^{-1} \chi$, where $\cF$ is the Fourier transform. Let $\lam > 0$ to be determined. We decompose $f$ into low and high frequency parts $f = f_l + f_h$ as follows
\[
f_l = \cF^{-1} ( \hat{f}(\xi)  \chi(\lam^{-1} \xi)  )  ,
  \quad f_h
  = \sum_{k \geq 1} f_k \teq \sum_{k \geq 1}  
   \cF^{-1} ( \hat{f}(\xi) (  \chi_{2^k \lam}(\xi) - \chi_{2^{k-1}\lam} (\xi) ),
\]
where $\hat{g} $ denotes the Fourier transform of $g$ and $\chi_{\lam}(\xi) \teq \chi(\lam^{-1} \xi)$. The frequency of $f_l$ supports in $|\xi| \leq 2 \lam$ and $f_k$ supports in the annulus $ |\xi| \in [2^{k-1} \lam, 2^{k+1}\lam]$.

Applying the Bernstein inequality (see e.g. \cite{bahouri2011fourier}) and then the Young's inequality for convolution, we obtain 
\[
\bal
|| f_l ||_{L^{1+ 2/q}} 
&\les \lam^{ \f{q}{2} - \f{1}{1 + 2/q}}     || f_l ||_{L^{2/q}} 
\les \lam^{ \f{q}{2} - \f{q}{q + 2}}     || f ||_{L^{2/q}} 
= \lam^{\f{q^2}{2(q+2)}}  || f ||_{L^{2/q}} 
, \\
|| f_k ||_{L^{1 + 2/q}}
&\les (2^k\lam)^{ -\f{1}{2} + \f{1}{2} - \f{1}{1 + 2/q} } || \Lam f_k ||_{L^2}
\les (2^k\lam)^{ -\f{q}{q+2}} || f||_{\dot{H}^{1/2}},
\eal
\]
where the constants are absolute since $q \in [1, \f{5}{4}]$. Applying the triangle inequality yields
\[
|| f||_{L^{1+ 2/q}} 
\les || f_l ||_{L^{1+ 2/q} } + 
\sum_{k \geq 1} || f_k ||_{L^{1 + 2/q}}
\les 
\lam^{\f{q^2}{2(q+2)}}  || f ||_{L^{2/q}} 
 + \lam^{ -\f{q}{q+2}} || f||_{\dot{H}^{1/2}}.
\]
Optimizing $\lam$ (balancing two terms in the upper bound) as follows 
\[
\lam =   || f||_{L^{2/q}}^{-2 / q}  || f||_{\dot{H}^{1/2}}^{2/q}
\]
completes the proof.
\end{proof}

\bibliographystyle{plain}
\bibliography{selfsimilar}

\begin{thebibliography}{10}

\bibitem{bahouri2011fourier}
Hajer Bahouri, Jean-Yves Chemin, and Rapha{\"e}l Danchin.
\newblock {\em Fourier analysis and nonlinear partial differential equations},
  volume 343.
\newblock Springer Science \& Business Media, 2011.

\bibitem{beale1984remarks}
JT~Beale, T~Kato, and A~Majda.
\newblock Remarks on the breakdown of smooth solutions for the 3-{D} {E}uler
  equations.
\newblock {\em Communications in Mathematical Physics}, 94(1):61--66, 1984.

\bibitem{caffarelli1982partial}
L~Caffarelli, R~Kohn, and L~Nirenberg.
\newblock Partial regularity of suitable weak solutions of the
  {N}avier-{S}tokes equations.
\newblock {\em Communications on Pure and Applied Mathematics}, 35(6):771--831,
  1982.

\bibitem{Cor10}
A~Castro and D~C\'ordoba.
\newblock Infinite energy solutions of the surface quasi-geostrophic equation.
\newblock {\em Advances in Mathematics}, 225(4):1820--1829, 2010.

\bibitem{chen2019finite}
Jiajie Chen, Thomas~Y Hou, and De~Huang.
\newblock On the finite time blowup of the de gregorio model for the 3d euler
  equation.
\newblock {\em arXiv preprint arXiv:1905.06387}, 2019.

\bibitem{CLM85}
P~Constantin, P.~D. Lax, and A.~Majda.
\newblock A simple one‐dimensional model for the three‐dimensional
  vorticity equation.
\newblock {\em CPAM}, 38(6):715--724, 1985.

\bibitem{constantin2015long}
Peter Constantin, Andrei Tarfulea, and Vlad Vicol.
\newblock Long time dynamics of forced critical sqg.
\newblock {\em Communications in Mathematical Physics}, 335(1):93--141, 2015.

\bibitem{constantin2012nonlinear}
Peter Constantin and Vlad Vicol.
\newblock Nonlinear maximum principles for dissipative linear nonlocal
  operators and applications.
\newblock {\em Geometric And Functional Analysis}, 22(5):1289--1321, 2012.

\bibitem{cordoba2005formation}
A~C{\'o}rdoba, D~C{\'o}rdoba, and MA~Fontelos.
\newblock Formation of singularities for a transport equation with nonlocal
  velocity.
\newblock {\em Annals of Mathematics}, pages 1377--1389, 2005.

\bibitem{cordoba2004maximum}
Antonio C{\'o}rdoba and Diego C{\'o}rdoba.
\newblock A maximum principle applied to quasi-geostrophic equations.
\newblock {\em Communications in mathematical physics}, 249(3):511--528, 2004.

\bibitem{DG90}
S~De~Gregorio.
\newblock On a one-dimensional model for the three-dimensional vorticity
  equation.
\newblock {\em Journal of Statistical Physics}, 59(5-6):1251--1263, 1990.

\bibitem{DG96}
S~De~Gregorio.
\newblock A partial differential equation arising in a 1d model for the 3d
  vorticity equation.
\newblock {\em Mathematical Methods in the Applied Sciences},
  19(15):1233--1255, 1996.

\bibitem{dong2008well}
Hongjie Dong.
\newblock Well-posedness for a transport equation with nonlocal velocity.
\newblock {\em Journal of Functional Analysis}, 255(11):3070--3097, 2008.

\bibitem{Elg17}
T.~M. Elgindi and I.~J. Jeong.
\newblock On the effects of advection and vortex stretching.
\newblock {\em arXiv preprint arXiv:1701.04050}, 2017.

\bibitem{elgindi2019finite}
Tarek~M Elgindi.
\newblock Finite-time singularity formation for $c^{1,\alpha}$ solutions to the
  incompressible euler equations on $\mathbb{R}^3$.
\newblock {\em arXiv preprint arXiv:1904.04795}, 2019.

\bibitem{Elg19}
Tarek~M Elgindi, Tej-Eddine Ghoul, and Nader Masmoudi.
\newblock Stable self-similar blowup for a family of nonlocal transport
  equations.
\newblock {\em arXiv preprint arXiv:1906.05811}, 2019.

\bibitem{ferreira2018global}
Lucas~CF Ferreira and Valter~VC Moitinho.
\newblock Global smoothness for a 1d supercritical transport model with
  nonlocal velocity.
\newblock {\em arXiv preprint arXiv:1809.04373}, 2018.

\bibitem{Har52}
GH~Hardy, JE~Littlewood, and G~P{\'o}lya.
\newblock {\em Inequalities}.
\newblock Cambridge university press, 1952.

\bibitem{hou2018global}
Thomas~Y Hou, Pengfei Liu, and Fei Wang.
\newblock Global regularity for a family of 3d models of the axi-symmetric
  navier--stokes equations.
\newblock {\em Nonlinearity}, 31(5):1940, 2018.

\bibitem{lei2009stabilizing}
TY~Hou and Z~Lei.
\newblock On the stabilizing effect of convection in three-dimensional
  incompressible flows.
\newblock {\em Communications on Pure and Applied Mathematics}, 62(4):501--564,
  2009.

\bibitem{hou2008dynamic}
TY~Hou and C~Li.
\newblock Dynamic stability of the three-dimensional axisymmetric
  {N}avier-{S}tokes equations with swirl.
\newblock {\em Communications on Pure and Applied Mathematics}, 61(5):661--697,
  2008.

\bibitem{kiselev2018}
A~Kiselev.
\newblock Small scales and singularity formation in fluid dynamics.
\newblock {\em arXiv:1807.00184 [math.AP]}, 2018.

\bibitem{kiselev2007global}
Alexander Kiselev, Fedor Nazarov, and Alexander Volberg.
\newblock Global well-posedness for the critical 2d dissipative
  quasi-geostrophic equation.
\newblock {\em Inventiones mathematicae}, 167(3):445--453, 2007.

\bibitem{landman1988rate}
MJ~Landman, GC~Papanicolaou, C~Sulem, and PL~Sulem.
\newblock Rate of blowup for solutions of the nonlinear {S}chr{\"o}dinger
  equation at critical dimension.
\newblock {\em Physical Review A}, 38(8):3837, 1988.

\bibitem{lei2018constantin}
Zhen Lei, Jie Liu, and Xiao Ren.
\newblock On the constantin-lax-majda model with convection.
\newblock {\em arXiv preprint arXiv:1811.09754}, 2018.

\bibitem{martel2014blow}
Yvan Martel, Frank Merle, Pierre Rapha{\"e}l, et~al.
\newblock Blow up for the critical generalized korteweg--de vries equation. i:
  Dynamics near the soliton.
\newblock {\em Acta Mathematica}, 212(1):59--140, 2014.

\bibitem{mclaughlin1986focusing}
DW~McLaughlin, GC~Papanicolaou, C~Sulem, and PL~Sulem.
\newblock Focusing singularity of the cubic {S}chr{\"o}dinger equation.
\newblock {\em Physical Review A}, 34(2):1200, 1986.

\bibitem{merle1997stability}
Frank Merle and Hatem Zaag.
\newblock Stability of the blow-up profile for equations of the type ut=∆ u+|
  u| p- 1u.
\newblock {\em Duke Math. J}, 86(1):143--195, 1997.

\bibitem{merle2015stability}
Frank Merle and Hatem Zaag.
\newblock On the stability of the notion of non-characteristic point and
  blow-up profile for semilinear wave equations.
\newblock {\em Communications in Mathematical Physics}, 333(3):1529--1562,
  2015.

\bibitem{OSW08}
H~Okamoto, T~Sakajo, and M~Wunsch.
\newblock On a generalization of the constantin–lax–majda equation.
\newblock {\em Nonlinearity}, 21(10):2447--2461, 2008.

\bibitem{Schonet1986}
S.~Schochet.
\newblock Explicit solutions of the viscous model vorticity equation.
\newblock {\em Communications on Pure and Applied Mathematics}, 39(4):531--537,
  1986.

\bibitem{silvestre2014transport}
L~Silvestre and V~Vicol.
\newblock On a transport equation with nonlocal drift.
\newblock {\em transactions of the American Mathematical Society},
  368(9):6159--6188, 2016.

\bibitem{sverak2017certain}
Vladimir Sverak.
\newblock On certain models in the pde theory of fluid flows.
\newblock {\em Journ{\'e}es {\'E}quations aux d{\'e}riv{\'e}es partielles},
  2017.

\bibitem{wunsch2011generalized}
Marcus Wunsch.
\newblock The generalized constantin-lax-majda equation revisited.
\newblock {\em Communications in Mathematical Sciences}, 9(3):929--936, 2011.

\end{thebibliography}

\end{document}